\documentclass[a4paper,11pt]{article}

\usepackage{color,url}
\usepackage{enumitem}
\usepackage{graphicx}

\usepackage{psfrag}
\usepackage{epic}
\usepackage{amssymb}
\usepackage{afterpage}
\usepackage{epsfig}
\usepackage{pstricks}
\usepackage{stmaryrd,url}
\usepackage{tikz,amsmath}
\usepackage{diagbox}
\usepackage{framed,multirow}
\usepackage{epstopdf}

\usepackage{amsmath}  
\usepackage{amssymb}   
\usepackage{amsthm} 

\numberwithin{equation}{section}

\usepackage[top=1.0in,bottom=1.0in,left=1.0in,right=1.0in]{geometry}
\usepackage[font=scriptsize,labelfont=bf,width=0.85\textwidth]{caption}
\usepackage{subcaption}
\usepackage[colorlinks=true,linkcolor=blue]{hyperref}
\usepackage{makecell}
\usepackage{framed,multirow}
\usepackage{diagbox}

\usepackage{algorithmic,algorithm,xpatch}
\xpatchcmd{\algorithmic}{\setcounter}{\algorithmicfont\setcounter}{}{}
\providecommand{\algorithmicfont}{}
\providecommand{\setalgorithmicfont}[1]{\renewcommand{\algorithmicfont}{#1}}

\newlength{\bibitemsep}
\newlength{\bibparskip}
\let\oldthebibliography\thebibliography
\renewcommand\thebibliography[1]{%
  \oldthebibliography{#1}%
  \setlength{\parskip}{\bibitemsep}%
  \setlength{\itemsep}{\bibparskip}%
}

\newtheorem{truth}{Theorem}

\newtheorem{remark}{Remark}
\newtheorem{lemma}{Lemma}

\title{A discontinuous Galerkin method for nonlinear biharmonic Schr\"{o}dinger equations}

\begin{document}
\author{Lu Zhang 
\thanks{Department of Applied Physics and Applied Mathematics, Columbia University, New York, NY 10027, USA. Email: lz2784@columbia.edu}}
\maketitle

\begin{abstract}
This paper proposes and analyzes a fully discrete scheme that discretizes space with an ultra-weak local discontinuous Galerkin scheme and time with the Crank--Nicolson method for the nonlinear biharmonic Schr\"{o}dinger equation. We first rewrite the problem into a system with a second-order spatial derivative and then apply the ultra-weak discontinuous Galerkin method to the system. The proposed scheme is more computationally efficient compared with the local discontinuous Galerkin method because of fewer auxiliary variables, and unconditionally stable without any penalty terms; it also preserves the mass and Hamiltonian conservation that are important properties of the nonlinear biharmonic Schrödinger equation. We also derive optimal $L^2$-error estimates of the semi-discrete scheme that measure both the solution and the auxiliary variable with general nonlinear terms.  Several numerical studies demonstrate and support our theoretical findings. 
\end{abstract}

\textbf{Keywords}: discontinuous Galerkin,  nonlinear biharmonic Schr\"{o}dinger equation, stability, error estimates

\textbf{AMS subject }: 65M12, 	65M60

\section{Introduction}
Nonlinear biharmonic Schr\"{o}dinger equations arise commonly from the study of the propagation of intense laser beams, quantum mechanics, nonlinear optical fibers, propagation of electromagnetic beams in plasma, and many other applications (see \cite{lam1977self, agrawal2000nonlinear, ben2000dispersion, cui2007well,karpman1996stabilization, karpman2000stability, pausader2009cubic} and references therein). In this paper, we consider the nonlinear biharmonic Schr\"{o}dinger equations 
\begin{equation}\label{problem}
iu_{t} + \kappa\Delta^2 u + uF'(|u|^2) = 0, \ \ ({\bf x},t)\in \Omega\times(0,T],\quad \Omega\subseteq\mathbb{R}^d, \quad d = 1,2,
\end{equation}
with suitable boundary and initial conditions
, where $u({\bf x},t)$ is a complex function of space-time variable $({\bf x},t)$ with $|u|$ being its {Euclidean length}, {$i = \sqrt{-1}$ is the imaginary number}, $\Delta^2=\nabla ^4$ is the biharmonic operator, 
$F'(y) = \frac{dF}{dy}$ is a real function that measures the medium nonlinearity, and $\kappa$ is a real constant {dependent on the physical relevance}. Problem (\ref{problem}) is a special case of the nonlinear Schr\"{o}dinger equation in the following form with $\zeta = 0$
\begin{equation}\label{full_problem}
iu_{t} + \zeta \Delta u + \kappa \Delta^2u + uF'(|u|^2) = 0,
\end{equation}
which was introduced in \cite{karpman1996stabilization,karpman2000stability} to study the role of small fourth-order dispersion in the propagation of intense laser beams in a bulk medium.  For $F'(|u|^2) = |u|^{2\sigma}$ with $\sigma$ being a positive integer, there are many researchers in the past a few years concerned with local and global well-posedness and formation of singularities for both (\ref{problem}) and (\ref{full_problem}). For example, when $\kappa < 0$, \cite{karpman1996stabilization,karpman2000stability} have shown that the waveguide solutions of equation (\ref{full_problem}) are stable for all $\kappa < 0$ when $d\sigma \leq 2$ and also stable for $\kappa \ll -1$ when $2 <d\sigma<4$, but unstable for all $\kappa < 0$ when $\sigma d \geq 4$. (\ref{problem}) admits a very similar result with (\ref{full_problem}), \cite{ilan2002self, baruch2011singular} proved that for $\kappa > 0$, (\ref{problem}) is defocusing and exists globally, but when $\kappa < 0$, it is focusing, and there exists a critical exponent $d\sigma = 4$ that determines the blow-ups and global existence subject to the ($L^2$) size of the initial data. 

This paper presents and analyzes a fully discrete ultra-weak local discontinuous Galerkin scheme with the Crank--Nicolson time discretization for the nonlinear biharmonic Schr\"{o}dinger equation (\ref{problem}). The proposed scheme is implicit in time, unconditionally stable and preserves the mass and Hamiltonian associated with the problem at a discrete level.
 
Before proceeding further, we want to note that various numerical methods have been proposed to solve nonlinear Schr\"{o}dinger equations (\ref{full_problem}) for $\kappa = 0$ in literature, such as the finite difference methods \cite{akrivis1993finite,chang1995finite,besse2004relaxation,bao2013optimal}, the time-splitting pseudo-spectral methods \cite{bao2003numerical,robinson1993numerical,pathria1990pseudo}, the finite element methods \cite{akrivis1991fully,karakashian1999space,dag1999quadratic} and the discontinuous Galerkin (DG) methods \cite{xu2012optimal, chen2019ultra, zhang2021local}, to name a few. However, few numerical methods have been considered for problem (\ref{full_problem}) with $\kappa \neq 0$ in the literature. \cite{xiao2019conservative} proposed a conservative linearly-implicit difference scheme for the modified Zakharov system with high-order space fractional quantum correction, but the method converges only second-order in space. Zhang and Su \cite{zhang2021conservative} improve the convergence of the method to fourth-order in space by developing a linearly-implicit compact difference scheme when solving the Quantum Zakharov System. However, both schemes were only considered in the one-dimensional case. The Schrodinger equation in multiple dimensions has many applications, such as optimal observation or sensor location problems in piezoelectric actuators and damage detection.  Baruch et al. \cite{baruch2010singular,baruch2010ring} investigated singular solutions and ring-type singular solutions of (\ref{problem}) in the multidimensional case by adaptive grid methods and static grid redistribution methods, respectively. No numerical analysis is presented therein to test the numerical methods such as their stability, convergence, etc.  Nonetheless, we want to highlight that nonlinear Schr\"{o}dinger equations with higher-order dispersive term (\ref{full_problem}) in both one dimensional and multi-dimensional cases not only plays an important role in the physical model description, such as the quantum effect in the propagation of Langmuir waves in plasma, but also brings interesting mathematical effect, such as stabilization of soliton instabilities.  Therefore, it is worth, theoretically and practically, developing stable and efficient numerical methods for better understanding the dynamics within nonlinear biharmonic Schr\"{o}dinger equations. 

In this paper, we develop a stable and computationally favored ultra-weak local DG method to solve the nonlinear biharmonic Schr\"{o}dinger equation in both one dimension and two dimensions. The reason for us to establish DG methods is because of its flexibility in handling geometry, provable convergence properties, accommodating $h$-$p$ adaptivity, and high parallel efficiency.  DG method was first designed by Reed and Hill \cite{reed1973triangular} to solve a problem arising from first-order neutron transport subject to a conservation law.  This finite-element method applies a piecewise polynomial basis for both the numerical and test function, and it was originally designed to deal with the first spatial derivative only (see, e.g., \cite{reed1973triangular,cockburn1989tvb,cockburn1989kutta,cockburn1989tvb,cockburn1998runge} for detailed discussions). The original DG method has developed in several directions over the past few decades.  For instance, Cockburn and Shu \cite{cockburn1998local} proposed the so-called local discontinuous Galerkin (LDG) method to solve a wide class of nonlinear convection-diffusion equations with high-order spatial derivatives.  By introducing auxiliary variables that reduce the original problem into a lower-order system, typically with first-order spatial derivatives, the LDG methods ensure the stability of the scheme by suitable numerical fluxes embedded with the resulting system.  See \cite{dong2009analysis,xu2012optimal,yan2002local} and references therein for recent developments of the LDG method. Another streamline of development is motivated by the urge to solve high-order problems, and this includes the ultra-weak discontinuous Galerkin method (UWDG) introduced by \cite{despres1994formulation} for linear elliptic PDEs.  The idea of the UWDG method is to shift all the spatial derivatives through integration by parts to the test function in the weak formulation, and the stability of the scheme is guaranteed by certain numerical fluxes and additional internal penalty terms when necessary.  See \cite{cessenat1998application,shu2016discontinuous,cheng2008discontinuous,bona2013conservative} and the reference therein for the application and further development of the UWDG method.  

In this work,  motivated by \cite{tao2020ultraweak, tao2021discontinuous}, we develop a DG method by combining the LDG and UWDG methods and then test this new hybrid scheme for the high-order nonlinear biharmonic Schr\"{o}dinger equations in the form of (\ref{problem}).  To this end, we introduce a second-order spatial derivative as an auxiliary variable to reduce the fourth-order problem to a system which is second-order in space.  This allows us to ensure the stability of the proposed scheme through integration by parts and a suitable choice of numerical fluxes.  Moreover, {compared with the LDG method, only one auxiliary variable is needed within the new approach and this reduces the memory requirement and the computational cost.  Furthermore, compared with the UWDG method, our approach guarantees its stability without requiring internal penalty terms, and this also improves the robustness of this new scheme.}    



The rest of the paper is organized as follows.  We first present the governing equations and their DG formulation in Section \ref{sec_formula}.  Then we study and discuss the stability of our proposed scheme.  Section \ref{sec_error_estimate} introduces some projection operators and derives the optimal $L^2$-error estimates for the semi-discrete scheme through an auxiliary equation and suitable numerical fluxes.  Then we present the fully discrete ultra-weak local DG coupled with the Crank--Nicolson time discretization and prove its mass and Hamiltonian conserving properties in Section \ref{sec_time_dis}.  An arbitrary high order spectral deferred correction time integrator is also presented in this section.  A few numerical experiments in both $1$D and $2$D that demonstrate and verify the theoretical findings are shown in Section \ref{sec_numerical}.  Section \ref{conclusion} draws a brief conclusion of this work.

\section{Semi-discrete DG Formulation and Stability}\label{sec_formula}
We choose $\kappa = 1$ in (\ref{problem}) and consider the following nonlinear biharmonic Schr\"{o}dinger equation  
\begin{equation}\label{schrodinger_4th}
i u_{t}  + \Delta^2u  = - uF'(|u|^2) ,\quad {\bf x}\in \Omega \subseteq \mathbb{R}^d,\quad t\geq 0,,\quad d = 1,2,
\end{equation}
subject to initial and periodic boundary conditions to be specified.  This choice of $\kappa = 1$ is made for demonstration simplicity, while the DG method developed here applies to the cases when $\kappa$ takes other values; moreover, our scheme cooperates with general boundary conditions, while the error estimates in Section \ref{sec_error_estimate} require different technical tools and might become more complicated. We further assume that $F(|u|^2) \geq 0$ to guarantee the positivity of the Hamiltonian associated with (\ref{schrodinger_4th}).

To derive a DG formulation for (\ref{schrodinger_4th}), we denote $w: = \Delta u$ and collect the following second-order system
\begin{equation}\label{system_2nd}
\left\{
\begin{aligned}
iu_t & = -\Delta w - uF'(|u|^2), \\
 w &=  \Delta u.
\end{aligned}
\right.
\end{equation}
Note that any solution to (\ref{system_2nd}) formally satisfies the conservation of mass and Hamiltonian
\[{M(t)=M(0),\ \  H(t)=H(0),\ \ \forall t>0},\]
where the mass and Hamiltonian are given by
\[M(t):= \int_\Omega |u|^2\ d{\bf x},\quad \text{and}\quad H(t) := \int_\Omega |w|^2 + F(|u|^2) \ d{\bf x}.\]


\subsection{Notations}\label{notation}
 Let $\Omega_h$ denote a tessellation of $\Omega$ with shape-regular elements $K$ and denote $\Gamma_h = \cup_{K\in\Omega_h}\partial K$ to be the union of the boundary faces of elements $K\in\Omega_h$. We further denote the diameter of $K$ by $h_K$ and $h = \max_{K} h_K$. For example, $K$ is an interval when $d = 1$; and a rectangle for Cartesian meshes when $d = 2$. On each element $K$, we approximate $(u,w)$ by $(u_h, w_h)$, each belonging to the following space 
\[V_h^q: = \{v_h({\bf x}, t), v_h({\bf x}, t)\in\mathcal{Q}^q(K), q\geq 1, {\bf x}\in K, t\geq 0, \forall K\in\Omega_h\},\]
where $\mathcal{Q}^q(K)$ is the space of tensor product of complex polynomials of degree at most $q\geq 1$ in each variable defined on $K$.

Specifically, in the one dimensional case, we have $\Omega_h = \cup_{j=1}^N[x_{j-\frac{1}{2}}, x_{j+\frac{1}{2}}]$, that is, $K = I_j = (x_{j-\frac{1}{2}}, x_{j+\frac{1}{2}})$, $j = 1,2,\cdots, N$. For any $\eta \in V_h^q$, we denote $\eta^-$ and $\eta^+$ to be the right and left limit values of $\eta$ at $x_{j+\frac{1}{2}}$, respectively. Let ${\bf n}_L = -1$ and ${\bf n}_R = 1$ be the outward unit normal to the left and right end of each sub-cell $I_j$, respectively, we then have the average and the jump at $x_{j+\frac{1}{2}}$ as follows
\[\{\eta\}_{j+\frac{1}{2}} = \frac{1}{2}(v_{j+\frac{1}{2}}^+ + v_{j+\frac{1}{2}}^-),\quad [\eta]_{j+\frac{1}{2}} = v_{j+\frac{1}{2}}^-{\bf n}_L + v_{j+\frac{1}{2}}^+{\bf n}_R.\]
In the two dimensional case, we have $\Omega_h = \cup_{kj}[x_{k-\frac{1}{2}}, x_{k+\frac{1}{2}}]\times[y_{j-\frac{1}{2}}, y_{j+\frac{1}{2}}]$, $k = 1,\cdots,N_x$, $j = 1,\cdots, N_y$. For this situation, $K =I_k\times I_j= (x_{k-\frac{1}{2}}, x_{k+\frac{1}{2}})\times(y_{j-\frac{1}{2}}, y_{j+\frac{1}{2}})$. Let $e$ be an interior edge shared by the ``left" and ``right" elements denoted by $K_L$ and $K_R$. The ``left" and ``right" can be uniquely defined for each $e$ according to any fixed rule. In this work, considering the rectangle for Cartesian meshes, we refer to left and bottom directions as ``left'' and right and top directions as ``right''.  Let $\zeta$ be a continuously differentiable scalar function on $K_L$ and $K_R$, and $\zeta^- :=(\zeta|_{K_L})|_e$, $\zeta^+ := (\zeta|_{K_R})|_e$ be the left and right traces, respectively. We then introduce the conventional notations for averages and jumps
\[\{\zeta\} = \frac{1}{2}(\zeta^- + \zeta^+), [\zeta] = \zeta^-{\bf n}_L + \zeta^+{\bf n}_R, \{\nabla \zeta\} = \frac{1}{2}(\nabla \zeta^- + \nabla \zeta^+), [\nabla \zeta] = \nabla \zeta^-\cdot {\bf n}_L + \nabla \zeta^+\cdot{\bf n}_R,\]
where ${\bf n}_L$ and ${\bf n}_R$ are the outward unit normals to $\partial K_L$ and $\partial K_R$, respectively. 


\subsection{Semi-discrete DG formulation}\label{semi_form}
To seek an approximation of the second-order system, on each element $K\in\Omega_h$ we choose test functions $\phi, \psi \in V_h^q$ and apply them to the first and the second equation in (\ref{system_2nd}), respectively.  An integration by parts leads us to the following integral system
\begin{equation}\label{DG_scheme1}
\int_{K}  i u_{ht}\phi + w_h\Delta \phi + u_hF'(|u_h|^2)\phi\ d{\bf x} 
=  \int_{\partial K} -\widetilde{\nabla w_h}\cdot{\bf n}\phi + \widetilde{w_h}\nabla \phi\cdot{\bf n}\  dS
\end{equation}
and
\begin{equation}\label{DG_scheme2}
\int_{K}w_h\psi - u_h\Delta\psi\ d{\bf x} 
= \int_{\partial K}\widehat{\nabla u_h}\cdot{\bf n}\psi -  \widehat{u_h}\nabla \psi\cdot{\bf n}\ dS,
\end{equation}
where $\widetilde{\nabla w_{h}}$, $\widetilde{w_h}$, $\widehat{\nabla u_{h}}$ and $\widehat{u_h}$ are numerical fluxes at element boundaries, and $\bf n$ represents the outward unit normal to $\partial K$.  {Note that $q = 0$ is not an option as it yields inconsistency in the scheme}. 
To complete the DG formulations, we specify the numerical fluxes $\widetilde{\nabla w_{h}}$, $\widetilde{w_h}$, $\widehat{\nabla u_{h}}$ and $\widehat{u_h}$ at the element boundaries by choosing
\begin{equation}\label{flux1}
\widehat{u_h} = \alpha_1 u_{h}^+ + (1-\alpha_1) u_{h}^-, \  \widetilde{\nabla w_{h}} = (1-\alpha_1)\nabla w_{h}^+ + \alpha_1 \nabla w_{h}^-,
\end{equation}
and
\begin{equation}\label{flux2}
\widehat{\nabla u_{h}} = \alpha_2 \nabla u_{h}^+ + (1-\alpha_2) \nabla u_{h}^-,\ \widetilde{w_h} = (1-\alpha_2) w_{h}^+ + \alpha_2 w_{h}^-, 
\end{equation}
where $0\leq \alpha_1, \alpha_2 \leq 1$.  In particular, when $\alpha_1 = 0 \ \mbox{or} \ 1$, $\alpha_2 = 0 \ \mbox{or} \ 1$, we have the alternating fluxes that are also compatible with the error estimates in Sections \ref{initial_error} and \ref{error_sec}. Denote
\begin{align}
\mathcal{B}_K^1(w_h, \phi) &:= \int_{K} w_h\Delta\phi \ d{\bf x} + \int_{\partial K} \widetilde{\nabla w_{h}}\cdot {\bf n} \phi - \widetilde{w_h}\nabla \phi\cdot{\bf n}\ dS,\label{B1}\\
\mathcal{B}_K^2(u_h, \psi) &:= \int_{K}  - u_h\Delta\psi\ d{\bf x} - \int_{\partial K} \widehat{\nabla u_{h}}\cdot{\bf n}\psi - \widehat{u_h}\nabla \psi\cdot{\bf n}\ dS.\label{B2}
\end{align}
We can further simplify the DG scheme (\ref{DG_scheme1})--(\ref{DG_scheme2}) to 
\begin{align}
    \int_K iu_{ht}\phi + u_hF'(|u_h|^2)\phi\ d{\bf x} + \mathcal{B}_K^1(w_h,\phi) &= 0,\label{DG_scheme1_v2}\\
    \int_K w_h\psi\ d{\bf x} + \mathcal{B}_K^2(u_h,\psi) &= 0.\label{DG_scheme2_v2}
\end{align}
The notations $\mathcal{B}_K^{1}$ and $\mathcal{B}_K^{2}$ will be used frequently in the rest of the content to simplify the presentation.


\subsection{Energy conservation}\label{sec:conservation}
We now prove that the proposed scheme formulated in (\ref{DG_scheme1_v2})-(\ref{DG_scheme2_v2}) conserves both the semi-discrete mass and the semi-discrete Hamiltonian.  In particular, we show that they imply the stability of the scheme as follows.

\begin{truth}\label{them1}
	(Stability) The solution of the DG scheme (\ref{DG_scheme1_v2})-(\ref{DG_scheme2_v2}) with numerical fluxes (\ref{flux1})-(\ref{flux2}) satisfies the following conservation laws
	\begin{align}
	M_h(t) &= \int_{\Omega_h} |u_h|^2\ d{\bf x} = M_h(0),\label{M}\\
	H_h(t) &= \int_{\Omega_h} |w_h|^2 + F(|u_h|^2) \ d{\bf x} = H_h(0).\label{E}
	\end{align}
\end{truth}
\begin{proof}
Let us first prove the mass conservation.  To this end, we choose $\phi = u_h^\ast$ in (\ref{DG_scheme1_v2}) and $\psi = w_h^\ast$ in (\ref{DG_scheme2_v2}), where ``$\ast$'' represents the complex conjugate.  Summing (\ref{DG_scheme1_v2}) multiplied by $-i$ and (\ref{DG_scheme2_v2}) multiplied by $i$ up over all elements $K$, we obtain
	\begin{align}
		-i\sum_K\mathcal{B}^1_K(w_h, u_h^\ast) - i\int_{\Omega_h} iu_{ht}u_h^\ast +  |u_h|^2F'(|u_h|^2) \ d{\bf x} &= 0,\label{eq1}\\
		i\sum_K\mathcal{B}^2_K(u_h, w_h^\ast) + \int_{\Omega_h} iw_hw_h^\ast\ d{\bf x} & = 0,\label{eq2}
	\end{align}
	where $\mathcal{B}_K^1$ and $\mathcal{B}_K^2$ are defined in (\ref{B1}) and (\ref{B2}), respectively. Let $F$ be the interelement boundary face shared by two neighboring elements ($F$ is the boundary point $x_{j+\frac{1}{2}}$ when $d = 1$; and the interior edge $e$ when $d = 2$), because of periodic boundary condition, we apply integration by parts and find
\begin{equation*}\label{A11}
-i\sum_K\mathcal{B}^1_K(w_h, u_h^\ast) = \int_{\Omega_h} i\nabla w_{h}\cdot\nabla u_{h}^\ast \ d{\bf x} -i \sum_F \int_{ F}\widetilde{\nabla w_{h}}\cdot[u_h^\ast]  - \widetilde{w_h}[\nabla u_{h}^\ast] +[w_h\nabla u_{h}^*]\ dS ,
\end{equation*}
and
\begin{equation*}\label{A21}
i\sum_K\mathcal{B}^2_K(u_h, w_h^\ast) = \int_{\Omega_h}i\nabla u_{h}\cdot\nabla w_{h}^\ast\ d{\bf x} -i\sum_F\int_{F} \widehat{\nabla u_{h}}\cdot[w_h^\ast] - \widehat{u_h}[\nabla w_{h}^\ast] + [u_h\nabla w_{h}^\ast]\ dS.
 \end{equation*}

Then, computing the complex conjugate of (\ref{eq1})-(\ref{eq2}) and adding them to the resulting two equations, we arrive at
\begin{align}\label{eq5}
0 =& i\int_{\Omega_h} |u_h|^2F'(|u_h|^2) \ d{\bf x} - i\int_{\Omega_h} |u_h|^2F'(|u_h|^2) \ d{\bf x}\nonumber\\
=&\frac{d}{dt}\int_{\Omega_h} |u_h|^2 dx-i\sum_K \mathcal{B}_K^1(w_h,u_h^\ast)-\mathcal{B}_K^1(w_h^\ast,u_h)-\mathcal{B}_K^2(u_h,w_h^\ast)+\mathcal{B}_K^2(u_h^\ast,w_h) \\
=& \frac{d}{dt}\int_{\Omega_h} |u_h|^2 dx + 2\mbox{Im}\sum_F\int_F \widetilde{\nabla w_{h}}\cdot[u_h^\ast] - \widetilde{w_h}[\nabla u_{h}^\ast] +[w_h \nabla u_{h}^*]\ dS \nonumber\\
&+ 2\mbox{Im}\sum_F\int_F\widehat{\nabla u_{h}}\cdot[w_h^\ast] - \widehat{u_h}[\nabla w_{h}^\ast] + [u_h\nabla w_{h}^\ast]\ dS. \nonumber 
\end{align}
  Further, by using the fact
	\begin{equation}\label{fact1}
	[ab] = \big((1-\alpha) a^- + \alpha a^+ \big) [b] + \big((1-\alpha) b^+ + \alpha b^- \big)[a],\ \ \alpha \in [0,1],
	\end{equation}
	we get
\begin{align}\label{eq3}
\mathcal{K}_1:
=& \sum_F\int_F [w_h\nabla u_{h}^*]+\widehat{\nabla u_{h}}\cdot[w_h^\ast] - \widetilde{w_h}[\nabla u_{h}^\ast]  \ dS\ \nonumber\\
=&\sum_F \int_F\big(\alpha_2w_{h}^- + (1-\alpha_2)w_{h}^+\big)[\nabla u_{h}^\ast] + \big((1-\alpha_2)\nabla u_{h}^{\ast,-} + \alpha_2\nabla u_{h}^{\ast,+}\big)\cdot[w_h] \nonumber\\
 &+\big(\alpha_2\nabla u_{h}^+ + (1-\alpha_2)\nabla u_{h}^-\big)\cdot[w_h^\ast]- \big((1-\alpha_2)w_{h}^+ + \alpha_2 w_{h}^-\big)[\nabla u_{h}^\ast]\ dS\nonumber\\
=& \sum_F\int_F 2\mbox{Re}\Big(\big((1-\alpha_2)\nabla u_{h}^- + \alpha_2\nabla u_{h}^+\big)\cdot [w_h^\ast]\Big) \ dS,
\end{align}
	and
	\begin{align}\label{eq4}
	\mathcal{K}_2 := &\sum_F \int_F [u_h\nabla w_{h}^\ast]+\widetilde{\nabla w_{h}}\cdot[u_h^\ast] - \widehat{u_h}[\nabla w_{h}^\ast]\ dS\nonumber\\
	=& \sum_F \int_F 2\mbox{Re}\Big(\big((1-\alpha_1)\nabla w_{h}^+ + \alpha_1\nabla w_{h}^-\big)\cdot[u_h^\ast] \Big)\ dS.
	\end{align}
These identities indicate both $\mathcal{K}_1$ and $\mathcal{K}_2$ are real numbers. Plugging (\ref{eq3})-(\ref{eq4}) into (\ref{eq5}) leads to
	\[\frac{d}{dt}\int_{\Omega_h} |u_h|^2\ d{\bf x} = 0,\]
	which yields the mass conservation (\ref{M}). 
	
To prove the Hamiltonian conservation (\ref{E}), we differentiate (\ref{DG_scheme2_v2}) against time and obtain
	\begin{equation}\label{DG_scheme2t}
	\int_{K}w_{ht}\psi \d{\bf x} + \mathcal{B}_K^2(u_{ht}, \psi) = 0,
	\end{equation}
where $\mathcal{B}^2_K$ is defined in (\ref{B2}). Now, by choosing $\phi = u_{ht}^\ast$ in (\ref{DG_scheme1_v2}), $\psi = w_h^\ast$ in (\ref{DG_scheme2t}) and summing them over all elements $K$, respectively, we get
\begin{align}
		\sum_K\mathcal{B}^1_K(w_h,u_{ht}^\ast) + \int_{\Omega_h} iu_{ht}u_{ht}^\ast +  u_hF'(|u_h|^2)u_{ht}^\ast\ d{\bf x} &= 0,\label{A13}\\
		\sum_K\mathcal{B}^2_K(u_{ht}, w_h^\ast) + \int_{\Omega_h} w_{ht}w_h^\ast\ d{\bf x} & = 0,\label{A23}
\end{align} 
where $\mathcal{B}^1_K$ is defined in (\ref{B1}).  Moreover, we integrate by parts and have
\begin{equation*}\label{A12}
\sum_K\mathcal{B}^1_K(w_h, u_{ht}^\ast)= \int_{\Omega_h}  - \nabla w_{h}\cdot\nabla u_{ht}^\ast \ d{\bf x} + \sum_F\int_F \widetilde{\nabla w_{h}}\cdot[u_{ht}^\ast] - \widetilde{w_h}[\nabla u_{ht}^\ast] +[w_h\nabla u_{ht}^\ast]\ dS, 
\end{equation*}
and 
\begin{equation*}\label{A22}
\sum_K\mathcal{B}^2_K(u_{ht}, w_h^\ast)= \int_{\Omega_h} \nabla u_{ht}\cdot\nabla w_{h}^\ast\ d{\bf x}
-\sum_F\int_F \widehat{\nabla u_{ht}}\cdot[w_h^\ast] - \widehat{u_{ht}}[\nabla w_{h}^\ast] + [u_{ht}\nabla w_{h}^\ast]\ dS
\end{equation*}


Then, computing the complex conjugate of (\ref{A13})-(\ref{A23}) and adding them with the resulting two equations yields
\begin{align}\label{dG}
\frac{d}{dt}\int_{\Omega_h} F(|u_h|^2)+ |w_h|^2\ d{\bf x} 
=& -\sum_K\mathcal{B}_K^1(w_h,u_{ht}^\ast) + \mathcal{B}_K^1(w_h^\ast,u_{ht}) + \mathcal{B}_K^2(u_{ht}, w_h^\ast) + \mathcal{B}_K^2(u_{ht}^\ast, w_h)\nonumber\\ 
= &- 2\mbox{Re}\Big(\sum_F\int_F  \widetilde{\nabla w_{h}}\cdot[u_{ht}^\ast] - \widetilde{w_h}[\nabla u_{ht}^\ast] +[w_h\nabla u_{ht}^\ast] \ dS\Big)\nonumber\\
&+ 2\mbox{Re}\Big(\sum_F\int_F\widehat{\nabla u_{ht}}\cdot[w_h^\ast] - \widehat{u_{ht}}[\nabla w_{h}^\ast] + [u_{ht}\nabla w_{h}^\ast]\ dS\Big).    
\end{align}
Using identity (\ref{fact1}) and the numerical fluxes (\ref{flux1})-(\ref{flux2}), we have
\begin{align}\label{eq6}
\mathcal{K}_3:=&\sum_F\int_F [w_h\nabla u_{ht}^\ast]-\widetilde{w_h}[\nabla u_{ht}^\ast] - \widehat{\nabla u_{ht}}\cdot[w_h^\ast]\ dS\nonumber\\
=& -2i\sum_F \int_F\mbox{Im}\Big(\big((1-\alpha_2)\nabla u_{ht}^- + \alpha_2\nabla u_{ht}^+\big)\cdot[w_h^\ast] \Big)\ dS,
\end{align}
and
\begin{align}\label{eq7}
K_4 :=& \sum_F \int_F-[u_{ht}\nabla w_{h}^\ast]+\widetilde{\nabla w_{h}}\cdot[u_{ht}^\ast] + \widehat{u_{ht}}[\nabla w_{h}^\ast] \ dS\nonumber\\
	=& 2i\sum_F\int_F \mbox{Im}\Big(\big((1-\alpha_1)\nabla w_{h}^+ + \alpha_1\nabla w_{h}^-\big)\cdot[u_{ht}^\ast] \Big)\ dS.
\end{align}
Again, these identities indicate both $\mathcal{K}_3$ and $\mathcal{K}_4$ are pure imaginary numbers.  Finally, substituting (\ref{eq6}) and (\ref{eq7}) into (\ref{dG}) leads us to 
\[\frac{d}{dt}\int_{\Omega_h} F(|u_h|^2)+ |w_h|^2\ d{\bf x} = 0,\]
and this establishes the Hamiltonian conservation (\ref{E}).
\end{proof}

\section{Error Estimates}\label{sec_error_estimate}
In this section, We proceed to derive error estimates of the DG scheme (\ref{DG_scheme1_v2})-(\ref{DG_scheme2_v2}) for the nonlinear biharmonic Schr\"{o}dinger equation (\ref{schrodinger_4th}). For simplicity of analysis, we only consider the following alternating fluxes with $\alpha_1 = \alpha_2 = 1$ in (\ref{flux1}) and (\ref{flux2}), that is,
\begin{equation}\label{fluxa}
\widehat{u_h} = u_{h}^+,\ \ \widetilde{\nabla w_{h}} =  \nabla w_{h}^-, \ \ \widehat{\nabla u_{h}} =  \nabla u_{h}^+ ,\ \ \widetilde{w_h} = w_{h}^-.
\end{equation}
However, the error analysis can be easily generated to other types of alternating fluxes. In Section \ref{projection}, we review some projections and inequalities that are essential for our proof.  Section \ref{prior} presents the 
{\it a priori} error estimates needed to evaluate the nonlinear terms. The error estimates in the $L^2$-norm are given from Section \ref{initial_error} to Section \ref{error_sec}.  In the estimates, we denote by $C$ a generic positive constant which is independent of $h$ but may vary from line to line.

\subsection{Projections}\label{projection}
For the one dimensional case $d = 1$, we define the Gauss--Radau projections $P^{\pm}_h$ into $V_h^q$ such that for any $u\in H^{q+1}(\Omega_h)$ and $q\geq 2$ 
\begin{align}
&\int_{I_j} (P^{\pm}_h u - u)v_h\ dx = 0,\ \ \forall v_h\in \mathcal{P}^{q-2}(I_j),\nonumber\\
& P_h^{+}u(x_{j-\frac{1}{2}}^+) = u(x_{j-\frac{1}{2}}), \ \ (P_h^{+}u)_x(x_{j-\frac{1}{2}}^+) = u_x(x_{j-\frac{1}{2}}),\label{proj1}\\
& P_h^{-}u(x_{j+\frac{1}{2}}^-) = u(x_{j+\frac{1}{2}}), \ \ (P_h^{-}u)_x(x_{j+\frac{1}{2}}^-) = u_x(x_{j+\frac{1}{2}}).\label{proj2}
\end{align}
When $q = 1$, the Gauss--Radau projections are defined only by (\ref{proj1}) and (\ref{proj2}). And for the two dimensional case $d = 2$, we define the Gauss--Radau projections to be
\[\Pi_h^{\pm} u := (P_{hx}^{\pm}\otimes P_{hy}^{\pm}) u,\]
where the subscripts $x,y$ indicate the application of the one-dimensional operators $P^{\pm}_h$ with respect to the $x$-direction and the $y$-direction, respectively.

For each projection, the following inequality holds (see e.g., \cite{ciarlet2002finite}) for any $u\in H^{k+1}(\Omega_h)$
\begin{equation}\label{projrel}
\|u - Q_hu\|_{L^2(\Omega_h)} + h\|u - Q_hu\|_{L^\infty(\Omega_h)} + h^{\frac{1}{2}}\|u - Q_hu\|_{L^2(\Gamma_h)} \leq Ch^{q+1},
\end{equation}
where $Q_h = P^{\pm}_h, \Pi_h^{\pm}$. 

\subsection{A priori error estimate}\label{prior} 
Let us denote
\begin{align*}
e_u &= u - u_h = u - \mathbb{P}_h^{+} u + \mathbb{P}_h^{+} u - u_h =: \eta_u + \xi_u,\\
e_w &= w - w_h = w - \mathbb{P}_h^{-} w + \mathbb{P}_h^{-} w - w_h =: \eta_w + \xi_w,\\
e_{ut} &= u_t - u_{ht} = u_t - \mathbb{P}_h^{+} u_t + \mathbb{P}_h^{+} u_t - u_{ht} =: \eta_{ut} + \xi_{ut},
\end{align*}
where $\mathbb{P}_h^{\pm} = P_h^{\pm}$ when $d = 1$, and $\mathbb{P}_h^{\pm} = \Pi_h^{\pm}$ when $d = 2$. To deal with the nonlinearity in problem (\ref{schrodinger_4th}), we make an a priori error estimate assumption 
\begin{equation}\label{prior3}
\|e_u\|_{L^2(K)} +\|e_{ut}\|_{L^2(K)} \leq h, 
\end{equation}
which will be verified in Section \ref{verification_priori}.
Further by the inverse inequality, we have
\begin{equation}\label{prior2}
\|e_u\|_{L^\infty(K)} +\|e_{ut}\|_{L^\infty(K)} \leq C,
\end{equation}
where the constant $C$ depends on the exact solution $u$ and the total time $T$, but not $h$. To obtain an optimal error estimate in the two dimensional case $d = 2$, we also need some superconvergence results of $\mathcal{B}^1_K$ and $\mathcal{B}^2_K$.
\begin{lemma}\label{lemmaB1}
\cite{tao2020ultraweak} Let $\mathcal{B}^1_K$ and $\mathcal{B}^2_K$ be defined by (\ref{B1}) and (\ref{B2}). We then have for $q\geq 1$
\begin{equation*}\label{lemma1A}
    \mathcal{B}^1_K(\eta_w, \phi) =0, \quad\mathcal{B}^2_K(\eta_u, \psi) = 0
\end{equation*}
for all $u, w\in \mathcal{P}^{q+2}(K)$, and $\phi, \psi \in \mathcal{Q}^k(K)$.
\end{lemma}
\begin{lemma}\label{lemmaB2}
\cite{tao2020ultraweak} Let $\mathcal{B}_1^K$ and $\mathcal{B}_2^K$ be defined by (\ref{B1}) and (\ref{B2}). We then have 
\begin{equation*}\label{lemma2A}
\begin{aligned}
&|\mathcal{B}^1_K(\eta_w, \phi)| \leq Ch^{q+2}\|w\|_{W^{2q+4,\infty}(K)}\|\phi\|_{L^2(K)},\\ &|\mathcal{B}^2_K(\eta_u, \psi)| \leq Ch^{q+2}\|u\|_{W^{2q+4,\infty}(K)}\|\psi\|_{L^2(K)},
\end{aligned}
\end{equation*}
where $\phi, \psi \in \mathcal{Q}^k(K)$, the constant $C$ is independent of $h$.
\end{lemma}


\subsection{Error estimates for initial conditions}\label{initial_error}
This section is devoted to the analysis of the initial error estimates, which plays an essential role in the proof of optimal error estimates of the DG scheme (\ref{DG_scheme1_v2})-(\ref{DG_scheme2_v2}). Motivated by \cite{li2020analysis,zhang2021local}, we choose an initial approximation by the solution of a linear steady-state problem as in the following lemma.

\begin{lemma}\label{lemma3}
Suppose that the numerical initial condition of the DG scheme (\ref{DG_scheme1_v2})-(\ref{DG_scheme2_v2}) is chosen as the DG approximation with numerical fluxes (\ref{fluxa}) to a linear steady-state problem
	\begin{equation}\label{steady_state}
	iu + \Delta^2 u = iu_0 + \Delta^2 u_{0},
	\end{equation}
	where $u_0$ is the initial value of $u$, and periodic boundary conditions are considered. Further, denote $w = \Delta u$, then the DG approximation for (\ref{steady_state}) is given as
	\begin{equation}\label{DG_schemes1}
	\int_{K}  i u_h\phi + w_h\Delta \phi - (iu_0 + \Delta^2u_{0})\phi\ d{\bf x}
	=  \int_{\partial K}-\widetilde{\nabla w_{h}}\cdot{\bf n} \phi  + \widetilde{w_h}\nabla \phi\cdot{\bf n}\ dS,
	\end{equation}
	and
	\begin{equation}\label{DG_schemes2}
	\int_{K}w_h\psi - u_h\Delta \psi\ d{\bf x}\ d{\bf x}
	= \int_{\partial K}\widehat{\nabla u_{h}}\psi - \widehat{u_h}\nabla \psi\ dS,
	\end{equation}
	for all $\phi, \psi \in V_h^q, q\geq1$.
	We then have the following optimal initial error estimates for time-dependent nonlinear Schr\"{o}dinger equation (\ref{schrodinger_4th})
	\begin{equation*}\label{esti_ic}
	\|\xi_u(x,0)\|_{L^2(\Omega_h)} +\|\xi_{ut}(x,0)\|_{L^2(\Omega_h)} +  \|\xi_w(x,0)\|_{L^2(\Omega_h)} \leq Ch^{q+1},
	\end{equation*}
	where $C$ is a positive constant depends on $q$, $\|F'\|_{W^{2,\infty}(\Omega_h)}$ and $\|u\|_{L^\infty(\Omega_h)}$, but not $h$.
\end{lemma}

\begin{proof}
Let us consider the DG approximation (\ref{DG_schemes1})-(\ref{DG_schemes2}) for the numerical initial condition first. Then we have the error identities 
\begin{align}
&\int_{K}  i e_{u}\phi \ d{\bf x} +\mathcal{B}_K^1(e_w,\phi)  = 0,\label{error_s1}\\
&\int_{K} e_w\psi \ d{\bf x} +  \mathcal{B}_K^2(e_u, \psi) = 0,\label{error_s2}
\end{align}
where $e_u = u(x,0) - u_h(x,0)$, $e_w = w(x,0) - w_h(x,0)$, and $\mathcal{B}_K^1, \mathcal{B}_K^2$ are defined in (\ref{B1}) and (\ref{B2}), respectively.
 Multiplying (\ref{error_s1}) by $-i$, (\ref{error_s2}) by $i$, choosing $\phi = \xi_u^\ast$, $\psi = \xi_w^\ast$, and taking the complex conjugate of the resulting two equations, then summing them over all elements $K$ yields
	\begin{multline}\label{ini1}
	2\sum_K \int_{K} |\xi_u|^2  +\mbox{Re}(\eta_u\xi_u^\ast)-  \mbox{Im}(\eta_w\xi_w^\ast)\ d{\bf x} \\
	+ i\sum_K-\mathcal{B}_K^1(\xi_w,\xi_u^\ast)+\mathcal{B}_K^2(\xi_u,\xi_w^\ast) + \mathcal{B}_K^1(\xi_w^\ast,\xi_u) - \mathcal{B}_K^2(\xi_u^\ast,\xi_w) \\
	+ i\sum_K-\mathcal{B}_K^1(\eta_w,\xi_u^\ast)+\mathcal{B}_K^2(\eta_u,\xi_w^\ast) + \mathcal{B}_K^1(\eta_w^\ast,\xi_u) - \mathcal{B}_K^2(\eta_u^\ast,\xi_w) = 0,
	\end{multline}
where we have used the relations $e_u = \xi_u + \eta_u$ and $e_w = \xi_w + \eta_w$. By a similar analysis as in the derivation of mass conservation (\ref{M}) in Section \ref{sec:conservation}, one has
\[\sum_K-\mathcal{B}_K^1(\xi_w,\xi_u^\ast)+\mathcal{B}_K^2(\xi_u,\xi_w^\ast) + \mathcal{B}_K^1(\xi_w^\ast,\xi_u) - \mathcal{B}_K^2(\xi_u^\ast,\xi_w) = 0.\]
Further combining the property of projection operators (\ref{projrel}) and numerical fluxes (\ref{fluxa}), one reduces (\ref{ini1}) into
	\[\left\{
	\begin{aligned}
2\int_{\Omega_h} |\xi_u|^2  +\mbox{Re}(\eta_u\xi_u^\ast)-  \mbox{Im}(\eta_w\xi_w^\ast)\ d{\bf x} &= 0,\quad d = 1,\\
2\int_{\Omega_h} |\xi_u|^2  +\mbox{Re}(\eta_u\xi_u^\ast)-  \mbox{Im}(\eta_w\xi_w^\ast)\ d{\bf x} &\leq  Ch^{q+2}(\|\xi_u({\bf x},0)\|_{L^2(\Omega_h)} + \|\xi_w({\bf x},0)\|_{L^2(\Omega_h)}), \quad d = 2,
\end{aligned}
\right.\]
where we have also used Lemma \ref{lemmaB1} and Lemma \ref{lemmaB2} for the derivation of the case $d = 2$. This yields
	\begin{equation}\label{s_est_1}
	\|\xi_u({\bf x},0)\|_{L^2(\Omega_h)}^2 \leq Ch^{q+1}\left(\|\xi_u({\bf x},0)\|_{L^2(\Omega_h)} + \|\xi_w({\bf x},0)\|_{L^2(\Omega_h)}\right).
	\end{equation}
	Similarly, choosing $\phi = \xi_u^\ast$ in (\ref{error_s1}), $\psi = \xi_w^\ast$ in (\ref{error_s2}), and taking the complex conjugate of the resulting two equations, then summing them over all elements $K$ gives rise to
	\begin{multline*}\label{ini2}
	2\sum_K \int_{K} |\xi_w|^2  -\mbox{Im}(\eta_u\xi_u^\ast)+  \mbox{Re}(\eta_w\xi_w^\ast)\ d{\bf x} \\
	+ \sum_K\mathcal{B}_K^1(\xi_w,\xi_u^\ast)+\mathcal{B}_K^2(\xi_u,\xi_w^\ast) + \mathcal{B}_K^1(\xi_w^\ast,\xi_u) + \mathcal{B}_K^2(\xi_u^\ast,\xi_w) \\
	+ \sum_K\mathcal{B}_K^1(\eta_w,\xi_u^\ast)+\mathcal{B}_K^2(\eta_u,\xi_w^\ast) + \mathcal{B}_K^1(\eta_w^\ast,\xi_u) + \mathcal{B}_K^2(\eta_u^\ast,\xi_w) = 0.
	\end{multline*}
	By using the property of projection operators (\ref{projrel}), numerical fluxes (\ref{fluxa}), a similar analysis as in the derivation of the Hamiltonian conservation (\ref{E}) in Section \ref{sec:conservation}, and the Lemma \ref{lemmaB1}-- \ref{lemmaB2}, we obtain
	
	\begin{equation}\label{s_est_2}
	\|\xi_w({\bf x},0)\|_{L^2(\Omega_h)}^2 \leq Ch^{q+1}\left(\|\xi_u({\bf x},0)\|_{L^2(\Omega_h)} + \|\xi_w({\bf x},0)\|_{L^2(\Omega_h)}\right).
	\end{equation}
	Combing (\ref{s_est_1})-(\ref{s_est_2}) and using Young's inequality, we arrive at
	\begin{equation}\label{s_est_3}
	\|\xi_u({\bf x},0)\|_{L^2(\Omega_h)} + \|\xi_w({\bf x},0)\|_{L^2(\Omega_h)} \leq Ch^{q+1}.
	\end{equation}

{Next, we estimate $\|\xi_{ut}({\bf x},0)\|_{L^2(\Omega_h)}$ by using the relation between the time dependent equation (\ref{DG_scheme1_v2}) and steady state equation (\ref{DG_schemes1}).  To this end, we start from the initial error equation for the time-dependent problem (\ref{DG_scheme1_v2})}
	\[\int_{K} ie_{ut}\phi\ d{\bf x} + \mathcal{B}_K^1(e_w,\phi) + \int_{K} u F'(|u|^2)\phi - u_hF'(|u_h|^2)\phi\ d{\bf x} = 0,\]
	where $\mathcal{B}_K^1$ is defined in (\ref{B1}).
	Subtracting this identity from (\ref{error_s1}) gives
	\[\int_{K} ie_{ut}\phi - ie_{u}\phi + u F'(|u|^2)\phi - u_hF'(|u_h|^2)\phi\ d{\bf x} = 0.\]
Multiplying the above equation by $i$, then choosing $\phi = \xi_{ut}^\ast$ and summing it over all elements $K$, we arrive at
	\begin{equation}\label{i_xiut}
	\int_{\Omega_h} |\xi_{ut}|^2\ d{\bf x} = \sum_K\int_{K} -\eta_{ut}\xi_{ut}^\ast + e_u\xi_{ut}^\ast +i\left(u F'(|u|^2)\xi_{ut}^\ast - u_hF'(|u_h|^2)\xi_{ut}^\ast\right)\ d{\bf x}.
	\end{equation}
For the right-hand side, we apply the property (\ref{projrel}) of the projection operator $\mathbb{P}_h^+$ and (\ref{s_est_3}) to estimate
\begin{equation}\label{e11}
		\Big|\sum_K\int_{K} -\eta_{ut}({\bf x},0)\xi_{ut}^\ast({\bf x},0) + e_u({\bf x},0)\xi_{ut}^\ast({\bf x},0)\ d{\bf x}\Big| \leq Ch^{q+1}\|\xi_{ut}({\bf x},0)\|_{L^2(\Omega_h)},
		\end{equation}
while for the nonlinear terms we have from the Taylor expansion that
 \begin{equation}\label{taylor}
 F'(|u_h|^2) = F'(|u|^2) + F''(|u|^2)\gamma+ \frac{1}{2}F'''(|\hat{u}|^2)\gamma^2,
 \end{equation}
 where $|\hat{u}|^2$ is between $|u|^2$ and $|u_h|^2$, and
 \begin{equation}\label{gamma}
 \gamma := |u_h|^2 - |u|^2  = u_hu_h^{\ast} - uu^\ast = (u - e_u)(u^\ast- e_u^\ast) - uu^\ast = |e_u|^2 - 2\mbox{Re}(u^\ast e_u).
 \end{equation}
 
 Finally, we obtain
 \begin{align}
 &\Big|\sum_K\int_{K} i\left(u F'(|u|^2) - u_hF'(|u_h|^2)\right)\xi_{ut}^\ast\ d{\bf x}\Big|\nonumber\\
 \leq &  \Big|\sum_K\int_{K} u\left(-F''(|u|^2)\gamma - \frac{1}{2}F'''(|\hat{u}|^2)\gamma^2\right)\xi_{ut}^\ast\ d{\bf x}\Big|\nonumber\\
 &+ \Big|\sum_K\int_{K} e_u\left(F'(|u|^2) + F''(|u|^2)\gamma + \frac{1}{2}F'''(|\hat{u}|^2)\gamma^2\right)\xi_{ut}^\ast\ d{\bf x}\Big|\nonumber\\
 \leq & Ch^{q+1}\|\xi_{ut}(x,0)\|_{L^2(\Omega_h)},\label{e12}
 \end{align}
where $C$ is dependent of $\|F'\|_{W^{2,\infty}(\Omega_h)}$, $\|u\|_{L^\infty(\Omega_h)}$, but not $h$. Finally, combining (\ref{i_xiut}), (\ref{e11}) with (\ref{e12}) yields $\|\xi_{ut}(x,0)\|_{L^2(\Omega_h)} \leq Ch^{q+1}$, and this completes the proof.
\end{proof}


\subsection{Optimal error estimates for $t > 0$}\label{error_sec}
We are now ready to present error estimates for the DG scheme (\ref{DG_scheme1_v2})-(\ref{DG_scheme2_v2}) with the numerical fluxes (\ref{fluxa}).  In particular, we shall show that the estimates are optimal in the $L^2$-norm.
\begin{truth}\label{them2}
Let $(u,w)$ be a smooth solution of system (\ref{system_2nd}), and $(u_h, w_h)\in V^q_h\times V_h^q, q\geq1$ be the numerical solution of the DG scheme (\ref{DG_scheme1_v2})-(\ref{DG_scheme2_v2}) with the smooth initial data computed by (\ref{steady_state}) along with periodic boundary conditions and the numerical fluxes (\ref{fluxa}), then we have the following error estimates :
	\begin{equation}\label{error_est}
	\|e_u({\bf x},T)\|_{L^2(\Omega_h)} + \|e_w({\bf x},T)\|_{L^2(\Omega_h)} + \|e_{ut}({\bf x},T)\|_{L^2(\Omega_h)} \leq Ch^{q+1},
	\end{equation}
	where $C$ is a positive constant that depends on $q$, $\|F'\|_{W^{3,\infty}(\Omega_h)}$, $\|u\|_{L^\infty(\Omega_h)}$, $\left|\left|u_t\right|\right|_{L^\infty(\Omega_h)}$ and $T$, but not $h$.
\end{truth}
\begin{proof}
Our proof consists of three steps: i) the estimate of $d\|\xi_u\|_{L^2(K)}/dt$, ii) the estimate of $d\|\xi_w\|_{L^2(K)}/dt$, and iii) the estimate of $d\|\xi_{ut}\|_{L^2(K)}/dt$.
	
\textbf{Step one}: 
By the DG scheme (\ref{DG_scheme1_v2})--(\ref{DG_scheme2_v2}), we have the following error equations for any $(\phi,\psi) \in V_h^q\times V_h^q$
	\begin{align}
	&\int_{K} e_{ut} \phi\ d{\bf x}-i \int_{K}uF'(|u|^2)\phi - u_hF'(|u_h|^2)\phi\ d{\bf x} -i \mathcal{B}_K^1(e_w,\phi) = 0,\label{erroru1}\\
	&\int_{K} i e_w\psi \ d{\bf x}  + i \mathcal{B}_K^2(e_u,\psi) = 0 ,\label{erroru2}
	\end{align}
where $\mathcal{B}_K^1$ and $\mathcal{B}_K^2$ are defined in (\ref{B1}) and (\ref{B2}), respectively. Choosing $\phi = \xi_u^\ast$ and $\psi = \xi_w^\ast$ in (\ref{erroru1})-(\ref{erroru2}), and taking the complex conjugate, then adding them with the resulting two equations and summing over all elements $K$ yields

	\begin{equation}\label{xiu}
	\frac{d}{dt}\sum_K\int_{K} |\xi_u|^2\ dx = \Lambda_1 + \Lambda_2 + \Lambda_3 + \Lambda_4,
	\end{equation}
	where we have used the relations $e_u = \eta_u + \xi_u$, $e_w = \eta_w + \xi_w$, and
	\begin{align*}
	\Lambda_1 &:= \sum_K\int_{K} -2\mbox{Re}(\eta_{ut}\xi_u^\ast) \  + 2\mbox{Im}(\eta_w\xi_w^\ast)\ d{\bf x},\\
	\Lambda_2 &:= i\sum_K -\mathcal{B}_K^2(\xi_u,\xi_w^\ast) + \mathcal{B}_K^2(\xi_u^\ast,\xi_w) + \mathcal{B}_K^1(\xi_w,\xi_u^\ast) - \mathcal{B}_K^1(\xi_w^\ast,\xi_u), \\
	\Lambda_3 &:= i\sum_K -\mathcal{B}_K^2(\eta_u,\xi_w^\ast) + \mathcal{B}_K^2(\eta_u^\ast,\xi_w) + \mathcal{B}_K^1(\eta_w,\xi_u^\ast) - \mathcal{B}_K^1(\eta_w^\ast,\xi_u), \\
	\Lambda_4 &:= \sum_K\int_{K} -2F'(|u|^2)\mbox{Im}(u\xi_u^\ast)  +  2F'(|u_h|^2)\mbox{Im}(u_h\xi_u^\ast)\ d{\bf x},
	\end{align*}
 are to be estimated separately.
	
	From the property (\ref{projrel}) of the projection operators, we get
	\begin{equation}\label{lambda1}
	|\Lambda_1| \leq Ch^{q+1} \big(\|\xi_u\|_{L^2(\Omega_h)} + \|\xi_w\|_{L^2(\Omega_h)}\big),
	\end{equation}
	and the same analysis as in the derivation of the mass conservation (\ref{M}) in Theorem \ref{them1} leads us to 
	\begin{equation}\label{lambda2}
	    \Lambda_2 = 0,
	\end{equation}
	As for the estimation of $\Lambda_3$, from the definition of the projection operators $\mathbb{P}_h^{\pm}$, the numerical fluxes (\ref{fluxa}) and the Lemma \ref{lemmaB1}--\ref{lemmaB2}, we obtain
	\begin{equation}\label{lambda3}
\left\{
	\begin{aligned}
	&|\Lambda_3| = 0, \quad d = 1,\\
	&|\Lambda_3| \leq Ch^{q+2}(\|\xi_u\|_{L^2(\Omega_h)} + \|\xi_w\|_{L^2(\Omega_h)}), \quad d = 2.
	\end{aligned}
	\right.
	\end{equation}
	To estimate $\Lambda_4$, we first rewrite it as
	\begin{equation*}
	\Lambda_4 = -2\sum_K\int_K \big(F'(|u|^2) -F'(|u_h|^2) \big)\mbox{Im}(u\xi_u^\ast) +F'(|u_h|^2)\mbox{Im}(e_u\xi_u^\ast)\ d{\bf x}.
	\end{equation*}
	It can be further rewritten as follows thanks to (\ref{taylor}) and (\ref{gamma}) 
	\begin{equation*}
	\Lambda_4 = 2(\Lambda_{41} + \Lambda_{42} ),
	\end{equation*}
	where
	\[\Lambda_{41} = \sum_K \int_{K}\Big( F''(|u|^2)\gamma
	+ \frac{1}{2}F'''(|\hat{u}|^2)\gamma^2\Big)\mbox{Im}(u\xi_u^\ast)\ d{\bf x} \]
	and
	\[\Lambda_{42} = -\sum_K\int_{K}\Big(F'(|u|^2) + F''(|u|^2)\gamma+ \frac{1}{2}F'''(|\hat{u}|^2)\gamma^2\Big)\mbox{Im}(e_u\xi_u^\ast)\ d{\bf x}.\]
	We first have 
	\[\int_{K} |\gamma\mbox{Im}(u\xi_u^\ast)|  d{\bf x} \leq C\left(\|u\|_{L^\infty(K)} \|e_u\|_{L^\infty(K)} + \|u\|^2_{L^\infty(K)}\!\right)\left(\|\xi_u\|_{L^2(K)}^2 + \|\eta_u\|_{L^2(K)}^2\!\right),\]
	hence
	\begin{multline*}
	|\Lambda_{41}| \!\leq\! C\Big(\|F''\|_{L^\infty(\Omega_h)} + \|\gamma\|_{L^\infty(\Omega_h)}\|F'''\|_{L^\infty(\Omega_h)}\Big)\cdot\\
	\left(\|u\|_{L^\infty(\Omega_h)} \|e_u\|_{L^\infty(\Omega_h)} + \|u\|^2_{L^\infty(\Omega_h)}\!\right)
	\Big(\|\xi_u\|_{L^2(\Omega_h)}^2 + \|\eta_u\|_{L^2(\Omega_h)}^2\Big).
	\end{multline*}
	
	On the other hand, we have $\mbox{Im}(e_u\xi_u^\ast)  = \mbox{Im}(\eta_u\xi_u^\ast)$ since $\xi_u\xi_u^\ast = |\xi_u|^2$ is a real number, then 
	\begin{equation*}
	|\Lambda_{42}| \leq\! C\Big(\!\|F'\|_{L^\infty(\Omega_h)}\!+ \|\gamma\|_{L^\infty(\Omega_h)}\|F''\|_{L^\infty(\Omega_h)} \!+ \|\gamma^2\|_{L^\infty(\Omega_h)}\|F'''\|_{L^\infty(\Omega_h)}\!\Big)\Big(\|\xi_u\|_{L^2(\Omega_h)}^2 \!+\! \|\eta_u\|_{L^2(\Omega_h)}^2\Big).
	\end{equation*}
	
Combining the bounds of $\Lambda_{41}$ and $\Lambda_{42}$ derived above, we obtain
\begin{align}\label{lambda31}
|\Lambda_4| \leq C\Bigg(&\left(\|F''\|_{L^\infty(\Omega_h)} + \|\gamma\|_{L^\infty(\Omega_h)}\|F'''\|_{L^\infty(\Omega_h)}\right)\nonumber\\ &\cdot \left(\|u\|_{L^\infty(\Omega_h)} \|e_u\|_{L^\infty(\Omega_h)} + \|u\|^2_{L^\infty(\Omega_h)}\!\right) +\|F'\|_{L^\infty(\Omega_h)}
\\&+\|\gamma\|_{L^\infty(\Omega_h)}\|F''\|_{L^\infty(\Omega_h)} + \|\gamma^2\|_{L^\infty(\Omega_h)}\|F'''\|_{L^\infty(\Omega_h)}\Bigg)\Big(\|\xi_u\|_{L^2(\Omega_h)}^2 + \|\eta_u\|_{L^2(\Omega_h)}^2\Big),\nonumber
\end{align}

while from the a priori estimate (\ref{prior2}), we further get
\begin{equation}\label{gammae}
\|\gamma\|_{L^\infty(\Omega_h)} \leq C\left(1+\|u\|_{L^\infty(\Omega_h)}\right),\ \ \|\gamma^2\|_{L^\infty(\Omega_h)} \leq C\left(1 + \|u\|_{L^\infty(\Omega_h)} + \|u\|^2_{L^\infty(\Omega_h)}\right).
\end{equation}
Then we have from (\ref{gammae}) and (\ref{lambda31}) that
\begin{equation}\label{lambda4}
|\Lambda_4| \leq C \Big(\|\xi_u\|_{L^2(\Omega_h)}^2 + \|\eta_u\|_{L^2(\Omega_h)}^2\Big),
\end{equation}
where $C$ depends on $\|F'\|_{W^{2,\infty}(\Omega_h)}$ and $\|u\|_{L^\infty(\Omega_h)}$, but not $h$. Finally, plugging (\ref{lambda1})--(\ref{lambda3}) and (\ref{lambda4}) into (\ref{xiu}) and using Young's inequality, we have
	\begin{equation}\label{xiut}
	\frac{d}{dt}\|\xi_u\|_{L^2(\Omega_h)}^2 \leq C\Big(\|\xi_u\|_{L^2(\Omega_h)}^2 + \|\xi_w\|_{L^2(\Omega_h)}^2 + h^{2(q+1)} \Big).
	\end{equation}
	
\textbf{Step two}: 
We first differentiate (\ref{DG_scheme2_v2}) against time. Then we use the resulting equation and (\ref{DG_scheme1_v2}) to generate the following error equations
\begin{align}
&i\int_{K} e_{ut} \phi\ d{\bf x}+ \int_{K}uF'(|u|^2)\phi - u_hF'(|u_h|^2)\phi\ d{\bf x} + \mathcal{B}_K^1(e_w,\phi) = 0,\label{errorw1}\\
&\int_{K} e_{wt}\psi \ d{\bf x}  + \mathcal{B}_K^2(e_{ut},\psi) = 0 ,\label{errorw2}
\end{align}
for any $(\phi, \psi) \in V_h^q\times V_h^q$. Here $\mathcal{B}_K^1$ and $\mathcal{B}_K^2$ are defined in (\ref{B1}) and (\ref{B2}), respectively. Setting $\phi = \xi_{ut}^\ast$, $\psi = \xi_w^\ast$ in (\ref{errorw1})-(\ref{errorw2}), and taking the complex conjugate, then adding them with the resulting two equations and summing over all elements $K$ gives 
	\begin{equation}\label{xiw}
	\frac{d}{dt} \sum_K \int_{K} |\xi_w|^2\ d{\bf x} = \Theta_1 + \Theta_2 + \Theta_3 + \Theta_4,
	\end{equation}
	where we use the relations $e_u = \eta_u + \xi_u$, $e_w = \eta_w + \xi_w$ and denote
	\begin{align*}
	\Theta_1 &:= \sum_K\int_{K} 2\mbox{Im}(\eta_{ut}\xi_{ut}^\ast) - 2\mbox{Re}(\eta_{wt}\xi_w^\ast)\ d{\bf x},\\
	\Theta_2 &:= -\sum_K \mathcal{B}_K^1(\xi_w,\xi_{ut}^\ast) + \mathcal{B}_K^1(\xi^\ast_w,\xi_{ut})+ \mathcal{B}_K^2(\xi_{ut},\xi_w^\ast) + \mathcal{B}_K^2(\xi^\ast_{ut},\xi_w),\\
	\Theta_3 &:= -\sum_K \mathcal{B}_K^1(\eta_w,\xi_{ut}^\ast) + \mathcal{B}_K^1(\eta^\ast_w,\xi_{ut})+ \mathcal{B}_K^2(\eta_{ut},\xi_w^\ast) + \mathcal{B}_K^2(\eta^\ast_{ut},\xi_w),\\
	\Theta_4 &:= -\sum_K \int_{K}2F'(|u|^2)\mbox{Re}(u\xi_{ut}^\ast) - 2F'(|u_h|^2)\mbox{Re}(u_h\xi_{ut}^\ast)\ d{\bf x}. 
	\end{align*}
	By similar analysis as in step one, we get
	\[
	|\Theta_1| \leq Ch^{q+1}\Big(\|\xi_{ut}\|_{L^2(\Omega_h)} + \|\xi_w\|_{L^2(\Omega_h)}\Big),\quad |\Theta_4| \leq C \Big(\|\xi_{ut}\|_{L^2(\Omega_h)}^2 +\|\xi_{u}\|_{L^2(\Omega_h)}^2 + \|\eta_u\|_{L^2(\Omega_h)}^2\Big),\]
	and
	\[\left\{
	\begin{aligned}
	&|\Theta_3| = 0, \quad d = 1,\\
	&|\Theta_3| \leq  Ch^{q+2}(\|\xi_{ut}\|_{L^2(\Omega_h)} + \|\xi_w\|_{L^2(\Omega_h)}),\quad d = 2.
	\end{aligned}
	\right.\]
	From the same analysis as the derivation of the Hamiltonian conservation (\ref{E}) in Theorem \ref{them1}, we obtain
	\[\Theta_2 = 0 .\] 
	Plugging the above equality and inequalities into (\ref{xiw}) and using Young's inequality, we have
	\begin{equation}\label{xiwt}
	\frac{d}{dt}\|\xi_w\|_{L^2(\Omega_h)}^2 \leq C\Big(\|\xi_u\|_{L^2(\Omega_h)}^2 +\|\xi_{ut}\|_{L^2(\Omega_h)}^2+ \|\xi_w\|_{L^2(\Omega_h)}^2 + h^{2(q+1)} \Big).
	\end{equation}
	 Note that $\|\xi_{ut}\|_{L^2(\Omega_h)}^2$ appears in (\ref{xiwt}) and it is unknown. To get the estimate of $d\|\xi_w\|_{L^2(\Omega_h)}^2/dt$, we need to establish an inequality regarding $\|\xi_{ut}\|_{L^2(\Omega_h)}^2$. 
	
	\textbf{Step three} : an inequality for $d\|\xi_{ut}\|_{L^2(\Omega_h)}^2/dt$. Taking the time derivative of (\ref{DG_scheme1_v2})-(\ref{DG_scheme2_v2}) we get the following error equations
	\begin{align}
	&\int_{K} e_{utt} \phi\ d{\bf x}-i \int_{K}\frac{d}{dt}\Big(uF'(|u|^2)\phi - u_hF'(|u_h|^2)\Big)\phi\ dx -i \mathcal{B}_K^1(e_{wt},\phi) = 0,\label{errorut1}\\
	&\int_{K}i e_{wt}\psi \ d{\bf x}  + i\mathcal{B}_K^2(e_{ut},\psi) = 0 ,\label{errorut2}
	\end{align}
	for any $(\phi,\psi) \in V^q_h\times V^q_h$, and $\mathcal{B}_K^1$, $\mathcal{B}_K^2$ are defined in (\ref{B1}) and (\ref{B2}), respectively. Choosing $\phi = \xi_{ut}^\ast$, $\psi = \xi_{wt}^\ast$ in (\ref{errorut1})-(\ref{errorut2}), and taking their complex conjugate, then adding them with the resulting two equations and summing over all elements $K$, we have
	\begin{equation}\label{errorut}
	\frac{d}{dt}\sum_K\int_{K} |\xi_{ut}|^2\ d{\bf x} = W_1 + W_2 + W_3 + W_4 + W_5 + W_6,
	\end{equation}
	where we have used the relations $e_u = \eta_u + \xi_u$, $e_w = \eta_w + \xi_w$, and
	\begin{align*}
	W_1  &:= \sum_K\int_{K} -2\mbox{Re}(\eta_{utt}\xi_{ut}^\ast) \ d{\bf x},\\
	W_2 &:= i\sum_K \mathcal{B}_K^2(\xi_{ut}^\ast,\xi_{wt}) - \mathcal{B}_K^2(\xi_{ut},\xi_{wt}^\ast) + \mathcal{B}_K^1(\xi_{wt},\xi_{ut}^\ast) - \mathcal{B}_K^1(\xi_{wt}^\ast,\xi_{ut}), \\
	W_3 &:= i\sum_K \mathcal{B}_K^1(\eta_{wt},\xi_{ut}^\ast) - \mathcal{B}_K^1(\eta_{wt}^\ast,\xi_{ut}), \\
	W_4 &:= i\sum_K \!\!\int_{K}\frac{d}{dt}\Big(uF'(|u|^2) \!-\! u_hF'(|u_h|^2)\Big)\xi_{ut}^\ast \!-\! \frac{d}{dt}\Big(u^\ast F'(|u|^2) \!-\! u_h^{\ast}F'(|u_h|^2)\Big)\xi_{ut} \ d{\bf x},\\
	W_5 &: = i\sum_K \mathcal{B}_K^2(\eta_{ut}^\ast,\xi_{wt}) - \mathcal{B}_K^2(\eta_{ut},\xi_{wt}^\ast),\\
	W_6  &:= \sum_K\int_{K}  2\mbox{Im}(\eta_{wt}\xi_{wt}^\ast)d{\bf x},
	\end{align*}
	which will be estimated separately. By a similar analysis as in step one, we have
	\begin{equation}\label{w2}    
    |W_1| \leq Ch^{q+1} \|\xi_{ut}\|_{L^2(\Omega_h)},\quad W_2 = 0,
	\end{equation}
	and
	\begin{equation}\label{w2_2}  
	\left\{
	\begin{aligned}
	&|W_3 + W_5| = 0,\quad d = 1\\
	&|W_3| \leq Ch^{q+2}\|\xi_{ut}\|_{L^2(\Omega_h)}, \quad d = 2.
	\end{aligned}
	\right.
	\end{equation}
	 
	To estimate $W_4$, we first rewrite it as
	\begin{equation}\label{w4_first}
	W_4 = -2(W_{41} + W_{42}),
	\end{equation}
	where
	\begin{equation*}
W_{41}= \sum_K\int_{K}\big(F'(|u|^2)-F'(|u_h|^2)\big)\mbox{Im}(u_t\xi_{ut}^\ast) + F'(|u_h|^2)\mbox{Im}(e_{ut}\xi_{ut}^\ast)\ d{\bf x},
\end{equation*}
and
\begin{multline*}
W_{42} = \sum_K\int_{K}\left(F''(|u|^2)\frac{d|u|^2}{dt}-F''(|u_h|^2)\frac{d|u_h|^2}{dt}\right)\mbox{Im}(u\xi_{ut}^\ast) +F''(|u_h|^2)\frac{d|u_h|^2}{dt}\mbox{Im}(e_u\xi_{ut}^\ast)\ d{\bf x}.
\end{multline*}
 
The same analysis that leads to $\Lambda_4$ in step one gives rise to
\begin{equation}\label{W41_first}
	|W_{41}| \leq C
	\Big(\|\xi_u\|_{L^2(\Omega_h)}^2 + \|\xi_{ut}\|_{L^2(\Omega_h)}^2+ \|\eta_u\|_{L^2(\Omega_h)}^2 + \|\eta_{ut}\|_{L^2(\Omega_h)}^2\Big).
	\end{equation}

To estimate $W_{42}$, we find for some $|\underline{u}|^2$ between $|u|^2$ and $|u_h|^2$ that
\[F''(|u_h|^2) = F''(|u|^2) + F'''(|u|^2)\gamma+ \frac{1}{2}F^{(4)}(|\underline{u}|^2)\gamma^2,\]
where $\gamma$ is defined in (\ref{gamma}); moreover, we divide $W_{42}$ into two parts as
	\begin{equation}\label{w32}
	W_{42} = W_{42}^{(1)} + W_{42}^{(2)}, 
	\end{equation}
where
	\begin{equation*}
	W_{42}^{(1)} \!:=\!\! \sum_K\!\!\int_{K}\!\!\!\Bigg(\!\!-F''(|u|^2)\frac{d\gamma}{dt}-\left(\!\!F'''(|u|^2)\gamma \!+\! \frac{1}{2}F^{(4)}(|\underline{u}|^2)\gamma^2\!\right)
	\left(\!\frac{d|u|^2}{dt}\!+\!\frac{d\gamma}{dt}\right)\!\!\Bigg)\mbox{Im}(u\xi_{ut}^\ast)\ d{\bf x},
	\end{equation*}
	\begin{equation*}
	W_{42}^{(2)}\!: =\! \sum_{K}\int_{K}\left(F''(|u|^2) + F'''(|u|^2)\gamma + \frac{1}{2}F^{(4)}(|\underline{u}|^2)\gamma^2\right)\left(\frac{d|u|^2}{dt}+\frac{d\gamma}{dt}\right)\mbox{Im}(e_u\xi_{ut}^\ast)\ d{\bf x}.
	\end{equation*}
In light of the definition of $\gamma$ and estimate (\ref{prior2}), we obtain
	\begin{multline}\label{w311}
	|W_{42}^{(1)}| \leq C\Bigg(\|F''\|_{L^\infty(\Omega_h)}\Big(1+\|u\|_{L^\infty(\Omega_h)}+\left|\left|\frac{du}{dt}\right|\right|_{L^\infty(\Omega_h)}\Big)+\Big(\|F'''\|_{L^\infty(\Omega_h)}\left(1+\|u\|_{L^\infty(\Omega_h)}\right)\\ +\|F^{(4)}\|_{L^\infty(\Omega_h)}\|\gamma\|_{L^\infty(\Omega_h)}\left(1+\|u\|_{L^\infty(\Omega_h)}\right)\bigg)\bigg(\left|\left|\frac{d|u|^2}{dt}\right|\right|_{L^\infty(\Omega_h)} \!\!\\
	+\left|\left|\frac{d\gamma}{dt}\right|\right|_{L^\infty(\Omega_h)}\bigg)  \Bigg) \|u\|_{L^\infty(\Omega_h)}
	\left(\|\xi_{ut}\|_{L^2(\Omega_h)}^2 +\|\xi_u\|_{L^2(\Omega_h)}^2 + \|\eta_u\|_{L^2(\Omega_h)}^2+\|\eta_{ut}\|_{L^2(\Omega_h)}^2\right),
	\end{multline}
	and
	\begin{multline}\label{w312}
	|W_{42}^{(2)}| \leq C\left(\|F''\|_{L^\infty(\Omega_h)}+\|F'''\|_{L^\infty(\Omega_h)}\|\gamma\|_{L^\infty(\Omega_h)}+\|F^{(4)}\|_{L^\infty(\Omega_h)}\|\gamma^2\|_{L^\infty(\Omega_h)}\right)\\
	\cdot\left(\left|\left|\frac{d|u|^2}{dt}\right|\right|_{L^\infty(\Omega_h)} + \left|\left|\frac{d\gamma}{dt}\right|\right|_{L^\infty(\Omega_h)}\right)
	\left(\|\xi_{u}\|^2_{L^2(\Omega_h)} + \|\xi_{ut}\|_{L^2(\Omega_h)}^2 +\|\eta_u\|_{L^2(\Omega_h)}^2 \right).
	\end{multline}
	In addition, in light of the definition of $\gamma$ in (\ref{gamma}) and estimate (\ref{prior2}), we have
	\begin{equation}\label{gammate}
	\left|\left|\frac{d\gamma}{dt}\right|\right|_{L^\infty(\Omega_h)} \leq C\bigg(1 + \|u\|_{L^\infty(\Omega_h)} + \left|\left|\frac{du}{dt}\right|\right|_{L^\infty(\Omega_h)}\bigg).
	\end{equation}
	Plugging (\ref{gammae}) and (\ref{gammate}) into (\ref{w311})-(\ref{w312}), we apply (\ref{w4_first}) together with (\ref{W41_first}) and (\ref{w32})--(\ref{w312}) to obtain
	\begin{equation}\label{w3}
	|W_4| \leq C\left(\|\xi_{u}\|^2_{L^2(\Omega_h)} + \|\xi_{ut}\|_{L^2(\Omega_h)}^2+h^{2(q+1)}\right),
	\end{equation}
	where $C$ is a positive constant which depends on $\|F'\|_{W^{3,\infty}(\Omega_h)}$, $\|u\|_{L^\infty(\Omega_h)}$, and $\left|\left|\frac{d u}{d t}\right|\right|_{L^\infty(\Omega_h)}$, but not $h$. 
	Substituting (\ref{w2}), (\ref{w2_2}) and (\ref{w3}) into (\ref{errorut}) and using Young's inequality, we obtain
	\begin{equation}\label{xiut1}
	\frac{d}{dt} \|\xi_{ut}\|_{L^2(\Omega_h)}^2 \leq C\left(\|\xi_{u}\|^2_{L^2(\Omega_h)} + \|\xi_{ut}\|_{L^2(\Omega_h)}^2+h^{2(q+1)}\right) + |W_5| + |W_6|.
	\end{equation}
	Note that for the one-dimensional case, $d = 1$, $|W_3 + W_5| = 0$ as shown in (\ref{w2_2}), we will not have the term $|W_5|$ in (\ref{xiut1}). Collecting (\ref{xiut}), (\ref{xiwt}) and (\ref{xiut1}) leads us to
	\begin{multline*}
	\frac{d}{dt}\left(\|\xi_{u}\|_{L^2(\Omega_h)}^2 +\|\xi_{w}\|_{L^2(\Omega_h)}^2 +\|\xi_{ut}\|_{L^2(\Omega_h)}^2 \right)  \\ \leq C \left(\|\xi_{u}\|_{L^2(\Omega_h)}^2 +\|\xi_{w}\|_{L^2(\Omega_h)}^2+ \|\xi_{ut}\|_{L^2(\Omega_h)}^2 \right) + Ch^{2(q+1)} + |W_5| + |W_6|.
	\end{multline*}
    Now, we integrate this inequality from $0$ to $T$ to find that
\begin{align}\label{final1}
&\|\xi_{u}({\bf x},T)\|_{L^2(\Omega_h)}^2 +\|\xi_{w}({\bf x},T)\|_{L^2(\Omega_h)}^2 +\|\xi_{ut}({\bf x},T)\|_{L^2(\Omega_h)}^2 \nonumber\\  
\leq& \|\xi_{u}({\bf x},0)\|_{L^2(\Omega_h)}^2 +\|\xi_{w}({\bf x},0)\|_{L^2(\Omega_h)}^2 +\|\xi_{ut}({\bf x},0)\|_{L^2(\Omega_h)}^2 \nonumber\\
&+ C\int_0^T \|\xi_{u}\|_{L^2(\Omega_h)}^2 +\|\xi_{w}\|_{L^2(\Omega_h)}^2+ \|\xi_{ut}\|_{L^2(\Omega_h)}^2\  dt + Ch^{2(q+1)} + \int_0^T |W_5| + |W_6|\ dt.     
\end{align}  
For $\int_0^T |W_5| dt$ in the case of $d = 2$, we have from the integration by parts in time that
\begin{multline}\label{w5}
\int_0^T |W_5|\ dt = \int_0^T \Big|i\sum_K \mathcal{B}_K^2(\eta_{ut}^\ast,\xi_{wt}) - \mathcal{B}_K^2(\eta_{ut},\xi_{wt}^\ast)\Big|\ dt \\
\leq \sum_K \Big(\Big|\mathcal{B}_K^2(\eta_{ut}^\ast,\xi_{w})\Big| + \Big|\mathcal{B}_K^2(\eta_{ut},\xi_{w}^\ast)\Big|\Big)\Big|_0^T + \int_0^T \sum_K\Big|\mathcal{B}_K^2(\eta_{utt}^\ast,\xi_{w})\Big| + \Big|\mathcal{B}_K^2(\eta_{utt},\xi_{w}^\ast)\Big|\ dt.
\end{multline}
For $\int_0^T |W_6| dt$, again use integration by parts in time we obtain
	\begin{multline}\label{w6}
	\int_0^T |W_6|\ dt = \int_0^T\Big|\sum_K\int_{K} 2\mbox{Im}(\eta_{wt}\xi_{wt}^\ast) d{\bf x}\Big|\ dt\\
	 \leq \Big|\sum_K\int_{K} 2\mbox{Im}(\eta_{wt}\xi_{w}^\ast) d{\bf x}\Big|_0^T+ \int_0^T \Big|\sum_K\int_K 2\mbox{Im}(\eta_{wtt}\xi_{w}^\ast) \Big|dt\Big).
	\end{multline}
	Plugging (\ref{w5})--(\ref{w6}) into (\ref{final1}), we invoke Young's inequality, Lemma \ref{lemmaB1}, Lemma \ref{lemmaB2}, and Lemma \ref{lemma3} to find 
	\begin{align}\label{final2}
	&\|\xi_{u}({\bf x},T)\|_{L^2(\Omega_h)}^2 +\|\xi_{w}({\bf x},T)\|_{L^2(\Omega_h)}^2 +\|\xi_{ut}({\bf x},T)\|_{L^2(\Omega_h)}^2\nonumber\\
\leq& C\int_0^T \|\xi_{u}\|_{L^2(\Omega_h)}^2 +\|\xi_{w}\|_{L^2(\Omega_h)}^2+ \|\xi_{ut}\|_{L^2(\Omega_h)}^2\  dt + Ch^{2(q+1)} + \frac{1}{2}\|\xi_w({\bf x},T)\|_{L^2(\Omega_h)}^2.
	\end{align}
	Finally, applying Gronwall's inequality to (\ref{final2}) gives rise to
	\[\|\xi_{u}({\bf x},T)\|_{L^2(\Omega_h)}^2 +\|\xi_{w}({\bf x},T)\|_{L^2(\Omega_h)}^2 +\|\xi_{ut}({\bf x},T)\|_{L^2(\Omega_h)}^2 \leq Ch^{2(q+1)},\]
 and this collects the error estimate (\ref{error_est}) thanks to the triangle inequality and the property {(\ref{projrel})} of Gauss--Radau projection.
\end{proof}

\begin{remark}
{For the case $F'(\vert u\vert^2) \equiv \mbox{constant}$, one can obtain the same error estimates without assuming (\ref{prior3}). 
%
}
\end{remark}

\subsection{Verification of the a priori error estimate} \label{verification_priori}
We are now left to verify the a priori error estimate assumption (\ref{prior3}).  To see this, we first find that 
(\ref{prior3}) is true at $t = 0$ thanks to Lemma \ref{lemma3}.  To show it for all $t>0$, we argue by contradiction. Suppose that (\ref{prior3}) fails before $T$, there exist some $t_\ast\in(0,T)$ such that $t_\ast = \mbox{inf}\{t:\Vert(u - u_h)(\cdot,t)\Vert_{L^2(\Omega_h)}  + \Vert(u_t - u_{ht})(\cdot,t)\Vert_{L^2(\Omega_h)} > h\}$.
By the continuity of $\|(u-u_h)(\cdot,t)\|_{L^2(\Omega_h)} + \|(u_t-u_{ht})(\cdot,t)\|_{L^2(\Omega_h)}$, we have $h = \|(u - u_h )(\cdot,t_\ast )\|_{L^2(\Omega_h)} + \|(u_t - u_{ht})(\cdot,t_\ast )\|_{L^2(\Omega_h)}$. On the other hand, $(\ref{prior3})$ holds for $0 \leq t \leq t_\ast$, thus from Theorem \ref{them2}, we have $\|(u - u_h)(\cdot,t_\ast )\|_{L^2(\Omega_h)} +  \|(u_t - u_{ht})(\cdot,t_\ast )\|_{L^2(\Omega_h)}\leq Ch^{q+1}$
, which is a contradiction if $q \geq 1$. Therefore, we have $\|(u - u_h)(\cdot,t)\|_{L^2(\Omega_h)}  + \|(u_t - u_{ht})(\cdot,t)\|_{L^2(\Omega_h)}\leq h$ for all $0 \leq t \leq T$. Now we have completed the verification of (\ref{prior3}). 

\section{Time Discretization}\label{sec_time_dis}
{In this section, we extend the semi-discrete ultra-weak local DG method to the fully discrete method which also conserves the discrete mass and the discrete Hamiltonian.
	
\subsection{Crank--Nicolson time discretization}\label{sec_cn}
In this section, we discuss the Crank--Nicolson time scheme and show the mass and the Hamiltonian conservation properties of the corresponding fully time discrete scheme. Let $0 = t_0 < t_1 < \cdots < t_n < \cdots < \cdots < t_{N_t} = T$ and denote $h_t := t_{n+1} - t_n$. Here we use the uniform time step $h_t$ and denote by $u_h^n$ the DG solution at $t = t_n$. We also introduce the following two operators which will be used throughout the rest of the contents
\[ \mbox{central difference operator}: \delta u^n = \frac{u^{n+1} - u^{n-1}}{2h_t},\]
\[\mbox{average value operator}: \bar{\delta} u^n = \frac{u^{n+1} + u^{n-1}}{2}.\]
The fully discrete approximation $u^n_h = u(\cdot,t_n)$ of problem (\ref{system_2nd}) is given as follows

\begin{align}
	 \int_{K} i\delta u_h^n \phi + \bar{\delta} w_h^n \Delta \phi + \bar{\delta} u_h^n \mathcal{G}\phi\ d{\bf x} &=\int_{\partial K} -  \bar{\delta}\widetilde{\nabla w_{h}^n}\cdot{\bf n} \phi + \bar{\delta}\widetilde{w_h^n}\nabla\phi\cdot{\bf n}\ dS,\label{fully1} \\
	\int_{K} w^{n+1}_h \psi - u_h^{n+1}\Delta \psi\ d{\bf x} &= \int_{\partial K}\widehat{\nabla u_{h}^{n+1}}\cdot{\bf n}\psi  - \widehat{u_{h}^{n+1}}\nabla\psi\cdot{\bf n}\ dS,\label{fully2}\\
	 \int_{K} w^{n-1}_h \psi - u_h^{n-1}\Delta \psi\ d{\bf x} &= \int_{\partial K}\widehat{\nabla u_{h}^{n-1}}\cdot{\bf n}\psi - \widehat{u_{h}^{n-1}}\nabla \psi\cdot{\bf n}\ dS,\label{fully3}
	\end{align}
	for all test functions $\phi, \psi\in V_h^q,$ where $\mathcal{G} = \frac{F(|u_h^{n+1}|^2)- F(|u^{n-1}_h|^2)}{|u_h^{n+1}|^2 - |u_h^{n-1}|^2} $ and the numerical fluxes are defined in (\ref{fluxa}). We then have the following conservation property.
	\begin{truth}
	 For all $n$, the solution to the fully discrete ultra-weak LDG scheme (\ref{fully1})--(\ref{fully3}) conserves the discrete mass
		\begin{equation}\label{discrete_mass}
		M_h^{n+1} := \frac{1}{2}\sum_K\int_{K} \vert u_h^n\vert + \vert u_h^{n+1}\vert\ d{\bf x}, 
		\end{equation}
		and the discrete Hamiltonian
		\begin{equation}\label{discrete_H}
		H_h^{n+1} := \frac{1}{2}\sum_K\int_{K} \vert w_h^n\vert + \vert w_h^{n+1}\vert  + F(|u_h^{n+1}|^2) + F(|u_h^n|^2)\  d{\bf x}.
		\end{equation}
	\end{truth}
	\begin{proof}
		To prove the fully discrete mass conservation (\ref{discrete_mass}), we choose the test function $\phi = \bar{\delta} u_h^{n,\ast}$ in (\ref{fully1}) to obtain
		\begin{equation}\label{fully11}
		\int_{K} i\delta u_h^n \bar{\delta} u_h^{n,\ast} + \bar{\delta} w_h^n \Delta\bar{\delta}  u^{n,\ast} + \bar{\delta} u_h^n \mathcal{G} \bar{\delta} u_h^{n,\ast}\ d{\bf x}
		= \int_{\partial K}- \bar{\delta}\widetilde{\nabla w_{h}^n}\cdot{\bf n} \bar{\delta} u_h^{n,\ast}+ \bar{\delta}\widetilde{w_h^n}\nabla\bar{\delta}  u_{h}^{n,\ast}\cdot{\bf n}\ dS, 
		\end{equation}
		and the test function $\psi = \bar{\delta} w_h^{n,\ast}/2$ in (\ref{fully2}) and (\ref{fully3}) to generate
{
		\begin{equation}\label{full2_1}
	\int_{K} (w^{n+1}_h \bar{\delta} w_h^{n,\ast} - u_h^{n+1}\Delta\bar{\delta}  w_{h}^{n,\ast})/2 \ d{\bf x}
		=\int_{\partial K} (\widehat{\nabla u_{h}^{n+1}}\cdot{\bf n}\bar{\delta} w_h^{n,\ast} - \widehat{u_{h}^{n+1}}\nabla\bar{\delta} w_{h}^{n,\ast}\cdot{\bf n})/2\ dS,
		\end{equation}
		and
		\begin{equation}\label{fully3_1}
	\int_{K} (w^{n-1}_h \bar{\delta} w_h^{n,\ast} - u_h^{n-1}\Delta\bar{\delta}  w_{h}^{n,\ast})/2 \ d{\bf x}
		=\int_{\partial K} (\widehat{\nabla u_{h}^{n-1}}\cdot{\bf n}\bar{\delta} w_h^{n,\ast} - \widehat{u_{h}^{n-1}}\nabla\bar{\delta} w_{h}^{n,\ast}\cdot{\bf n})/2\ dS,
		\end{equation}}
		Adding (\ref{full2_1}) to (\ref{fully3_1}), we have
		\begin{equation}\label{fully23}
		\int_{K} \bar{\delta} w^n_h \bar{\delta} w_h^{n,\ast} - \bar{\delta} u^n_h\Delta\bar{\delta} w_{h}^{n,\ast} \ d{\bf x} 
		= \int_{\partial K} \bar{\delta}\widehat{\nabla u_{h}^n}\cdot{\bf n}\bar{\delta} w_h^{n,\ast} - \bar{\delta}\widehat{u_{h}^n}\nabla\bar{\delta} w_{h}^{n,\ast}\cdot{\bf n}\ dS
		\end{equation}
		For the resulting equations (\ref{fully11}) and (\ref{fully23}), by the same analysis that leads to the conservation of the semi-discrete mass (\ref{M}) in Section \ref{sec:conservation}, we can obtain
		\begin{equation}\label{discrete_M_eq}
		2\times \frac{1}{2h_t} \bigg[ \sum_K\int_{K} \frac{\vert u_h^{n+1}\vert + \vert u_h^{n}\vert}{2} - \frac{\vert u_h^n\vert + \vert u_h^{n-1}\vert}{2}\ d{\bf x}\bigg] = 0.
		\end{equation}
		Combining (\ref{discrete_M_eq}) and the definition of $M_h^n$ in (\ref{discrete_mass}), we have $M_h^{n+1} = M_h^{n}$ for all $n$. This illustrates
		the fully discrete mass is conserved by using the fully discrete scheme (\ref{fully1})--(\ref{fully3}).
		
		For the fully discrete Hamiltonian conservation (\ref{discrete_H}), we let the test function $\phi = \delta u_h^{n,\ast}$ in (\ref{fully1}) to get
		\begin{equation}\label{fully1_2}
		\int_{K} i\delta u_h^n \delta u_h^{n,\ast} + \bar{\delta} w_h^n \Delta \delta u_{h}^{n,\ast} + \bar{\delta} u_h^n \mathcal{G} \delta u_h^{n,\ast}\ d{\bf x}
		=\int_{\partial K} - \bar{\delta}\widetilde{\nabla w_{h}^n} \delta u_h^{n,\ast}+ \bar{\delta}\widetilde{w_h^n}\nabla\delta u_{h}^{n,\ast}\cdot{\bf n}\ dS,
		\end{equation}
		In (\ref{fully2}) and (\ref{fully3}), we choose the test function $\psi = \bar{\delta}w^{n,\ast}_h$ to obtain
		\begin{equation}
		\int_{K} w^{n+1}_h \bar{\delta} w_h^{n,\ast}- u_h^{n+1}\Delta\bar{\delta} w_{h}^{n,\ast}\ d{\bf x} 
		=\int_{\partial K}   \widehat{\nabla u_{h}^{n+1}}\cdot{\bf n} \bar{\delta} w_h^{n,\ast}- \widehat{u_{h}^{n+1}}\nabla \bar{\delta} w_{h}^{n,\ast}\cdot{\bf n}\ dS,\label{fully2_2}
		\end{equation}
		\begin{equation}
		\int_{K} w^{n-1}_h \bar{\delta} w_h^{n,\ast}- u_h^{n-1}\Delta\bar{\delta} w_{h}^{n,\ast}\ d{\bf x} 
		=\int_{\partial K}   \widehat{\nabla u_{h}^{n-1}}\cdot{\bf n} \bar{\delta} w_h^{n,\ast}- \widehat{u_{h}^{n-1}}\nabla \bar{\delta} w_{h}^{n,\ast}\cdot{\bf n}\ dS.\label{fully3_2}
		\end{equation}
		Subtracting (\ref{fully2_2}) from (\ref{fully3_2}) and dividing the resulting equation by $2h_t$ yields
		\begin{equation}\label{discrete_H_1}
		\int_{K} \delta w^{n}_h \bar{\delta} w_h^{n,\ast} - \delta u_h^{n}\Delta\bar{\delta} w_{h}^{n,\ast}\ d{\bf x}
		= \int_{\partial K} \delta\widehat{\nabla  u_{h}^{n}}\cdot{\bf n}\bar{\delta} w_h^{n,\ast} - \delta\widehat{u_{h}^{n}}\nabla \bar{\delta} w_{h}^{n,\ast}\cdot{\bf n}\ dS.
		\end{equation}
		For the resulting equations (\ref{fully1_2}) and (\ref{discrete_H_1}), by utilizing the same analysis for the conservation of the semi-discrete Hamiltonian (\ref{E}) in Section \ref{sec:conservation}, we arrive at
		\begin{multline}\label{fully_H1}
		2\times\frac{1}{2\Delta t} \Big[\sum_K\int_{K} \frac{\vert w_h^{n+1}\vert + \vert w_h^{n}\vert}{2} - \frac{\vert w_h^n\vert + \vert w_h^{n-1}\vert}{2} \\+ \frac{F(|u_h^{n+1}|^2) + F(|u_h^n|^2)}{2} - \frac{F(|u_h^{n}|^2) + F(|u_h^{n-1}|^2)}{2}\ d{\bf x}\Big] = 0.
		\end{multline}
		From (\ref{fully_H1}) and the definition of $H_h^n$ in (\ref{discrete_H}), we have $H_h^{n+1} = H_h^{n}$ for all $n$, and this verifies that the fully discrete Hamiltonian is conserved by using the fully discrete scheme (\ref{fully1})--(\ref{fully3}).
	\end{proof}
	
	Note that the fully discrete scheme (\ref{fully1}) -- (\ref{fully3}) results in the following nonlinear algebraic equation
	\[{\bf U}^{n+1} = \mathcal{L}({\bf U}^{n-1}, {\bf U}^{n+1}) + \mathcal{N}({\bf U}^{n-1}, {\bf U}^{n+1}),\]
	where ${\bf U}$ containing the degrees of freedom for $u_h$, $\mathcal{L}({\bf U}^{n-1}, {\bf U}^{n+1}) $ is a linear function of ${\bf U}^{n-1}, {\bf U}^{n+1}$, and $\mathcal{N}({\bf U}^{n-1}, {\bf U}^{n+1})$ is a nonlinear function with respect to ${\bf U}^{n-1}, {\bf U}^{n+1}$. In the implementation, we use Newton's method to find ${\bf U}^{n+1}$ for each time level $t_{n+1}$. Since the second order central difference is used on time discretization and we are mainly concerned the effect of the spatial discretization, we use the time step $h_t = \mbox{cfl}\times h^4$ to guarantee that the error will be dominated by the spatial discretization when using the Crank--Nicolson time integrator for the numerical experiments.
	
	In what follows, we also present another popular time-stepping algorithm for the semi-discrete problem and compare the mass and the Hamiltonian evolution history in the numerical experiments with the fully discrete scheme coupled with the Crank--Nicolson time scheme proposed in this section.
}

\subsection{The spectral deferred correction (SDC) time-stepping algorithm} 
{We now describe} an SDC method to solve the semi-discrete problem generated by scheme (\ref{DG_scheme1_v2})--(\ref{DG_scheme2_v2}).  This method builds on the low-order time-stepping scheme, and then iterative corrections on a defect equation to obtain the desired order of accuracy (see e.g., \cite{dutt2000spectral,hagstrom2007spectral}).  We extend \cite{zhang2021beam} by applying it to the nonlinear problems in this paper.  In what follows, we present the SDC algorithm for the problems with both linear and nonlinear terms for completeness.  An essential step for this purpose is to use an implicit method for the linear terms but an explicit method for the nonlinear terms. 

To illustrate this idea, let us consider a generic ODE system as follows
\begin{equation*}
\left\{
\begin{aligned}
&y_t = Ay + G(y), \ \ t\in(0,T],\\
&y(0) = y_0,
\end{aligned}
\right.
\end{equation*}
where $y_0, y(t)\in \mathbb{C}^l$, $A\in M_{l\times l}(\mathbb{C})$  and $G: \mathbb{C}^l \rightarrow \mathbb{C}^l$ is a nonlinear function. Suppose the time interval $[0,T]$ is partitioned into $N_t$ subintervals as $0 = t_0 < t_1 < \cdots < t_n < \cdots < \cdots < t_{N_t} = T$. Denote $h_t^n := t_{n+1} - t_n$ and $y_n := y(t_n)$, then our SDC time stepping algorithm proceeds as follows:

\setalgorithmicfont{\footnotesize}

\begin{algorithm}
	\caption{SDC time stepping algorithm }
	\begin{algorithmic}[1]
		\STATE Input: $y^1_{n,0} = y_n$, $t_n$, $t_{n+1}$, $m$, $J$
		\STATE Compute $m$ Gauss--Radau points $\tau_i\in (t_n, t_{n+1}]$ and set $\tau_0 = t_n$, $k_i = \tau_i - \tau_{i-1}$
		\FOR{$i = 1, \cdots, m$}
		\STATE Solve $y_{n,i}^1 = y_{n,i-1}^1 + k_i\left(Ay_{n,i}^1 + G(y_{n,{i-1}}^1)\right)$,\COMMENT{Compute the initial approximation}
		\ENDFOR
		\FOR{$j = 1, \cdots, J$}
		\STATE Initialize $\epsilon_{n,0}^j = 0$, $\delta_{n,0}^j = 0$
		\FOR{$i = 1, \cdots, m$}
		\STATE Compute $\epsilon_{n,i}^j = y_n - y_{n,i}^{j} + I_{n,0}^{n,i} (Ay^{j}(\tau)+G(y^j(\tau)))$
		\STATE Compute intermediate function value $\bar{y} = y_{n,i-1}^j + \delta_{n,i-1}^j$
		\STATE Compute intermediate nonlinear function value $\bar{G} = G(\bar{y})$
		\STATE Solve $\delta_{n,i}^{j} = \delta_{n,i-1}^j + k_iA\delta_{n,i}^j + (\epsilon_{n,i}^j - \epsilon_{n,i-1}^j) + k_i\left(G(\bar{y})- G(y_{n,i-1}^j)\right)$
		\STATE Update $y_{n,i}^{j+1} = y_{n,i}^j + \delta_{n,i}^j$
		\ENDFOR
		\ENDFOR
		\STATE \textbf{return} $y^{J+1}_{n,m}$
	\end{algorithmic}
\end{algorithm}
{In this algorithm, $I_{n,0}^{n,i}(Ay^j(\tau) + G(y^j(\tau)))$ denotes the integral of the $(m-1)$-th degree interpolating polynomial on the $m$ nodes $(\tau_i, Ay_{n,i}^j + G(y_{n,i}^j))_{i = 1}^m$ over the subinterval $[t_n, \tau_i]$, and it is the numerical quadrature approximation of $\int_{t_n}^{\tau_i} Ay^j(\tau) + G(y^j(\tau)) d\tau$.} When the SDC scheme is used, we set $m = 5$ and $J = 15$ so that the convergence order in time ($2m-1$) is larger than the convergence order in space ($4th$ order in space) and also use a uniform time step $h_t = 0.025$.


\section{Numerical Simulations}\label{sec_numerical}

In this section, we present several numerical experiments to illustrate and support the convergence of the proposed DG scheme in Section \ref{sec_formula}. Through these studies, We use a standard modal basis formulation and the alternating flux (\ref{fluxa}) for the conciseness of demonstration.


 \subsection{Linear problem in one dimensional space}\label{liner_sec}
We first consider the biharmonic Schr\"{o}dinger equation with $F'(|u|^2) = 1$, 
\begin{equation}\label{eg1}
iu_t + u_{xxxx} + u = 0, \ \ (x,t) \in (0, 4\pi)\times(0, 1],
\end{equation} 
subject to periodic boundary condition and initial condition $u(x,0) = \cos(x) + i \sin(x)$. Note that this PDE has the following exact solution 
\[u(x,t) = \cos(x + 2t) + i\sin(x + 2t).\]

We uniformly discretize the spatial interval through vertices $x_j = jh$, $j = 0,\cdots ,N$, $h = 4\pi/N$.   Throughout the studies we present results by considering the degree of the approximation space of $u_h$ and $w_h$ being $q = (1,2,3)$.

From Table \ref{linear_u} to Table \ref{linear_w}, we present the $L^2$ and $L^\infty$ errors for the real and imaginary parts of $u$ and $w$, respectively.  We also include the corresponding numerical orders of accuracy subject to the variation of $q$ and $N$. {There are several conclusions we can make out from these tables.}  First of all, the proposed scheme consistently gives the optimal $(q + 1)$-th order of accuracy across the choices of size $N$ and the error norms.  Moreover, there are infinitesimal differences between the $L^2$ errors of the real and imaginary parts of both $u$ and $w$.  Indeed, their differences are at the order of $10^{-12}$ and we skip presenting them herein.  
			
\begin{table}
 	\footnotesize
 	\begin{center}
 		\scalebox{1.0}{
 			\begin{tabular}{c c c c c c c c c c}
 				\hline
 				~ & ~ &  $\mbox{Re}(u)$ &~ &~&~& $\mbox{Im}(u)$& ~ & ~ & ~\\
 				\cline{3-6} \cline{7-10}
 				$q$ & $N$ & $L^2$ error & order & $L^\infty$ error & order  & $L^2$ error & order & $L^\infty$ error & order \\
 				\hline
 				1& 10 & 1.01e-00& --& 4.47e-01& -- &1.01e-00 & -- & 4.54e-01 & --\\
 				~& 20 & 3.28e-01& 1.62& 1.47e-01& 1.60& 3.28e-01 & 1.62 & 1.43e-01 & 1.67\\
 				~& 40 & 8.78e-02& 1.90& 3.71e-02& 1.99&8.78e-02  & 1.90 & 3.71e-02 & 1.94\\
 				~& 80 & 2.23e-02& 1.97& 9.20e-03& 2.01&2.23e-02  & 1.97 & 9.20e-03 & 2.01\\
 				~& 160 & 5.61e-03& 1.99& 2.27e-03& 2.02&5.61e-03  & 1.99 & 2.27e-03 & 2.02\\
 				~ & ~ & ~ & ~ & ~ & ~ & ~ & ~ & ~ & ~\\
 				2& 10 & 6.68e-02& --& 3.06e-02& -- &6.68e-02 & -- & 2.98e-02 & --\\
 				~& 20 & 7.57e-03& 3.14& 3.89e-03& 2.97& 7.57e-03 & 3.14 & 3.76e-03 & 2.99\\
 				~& 40 & 9.24e-04& 3.03& 5.19e-04& 2.91&9.24e-04  & 3.03 & 5.19e-04 & 2.86\\
 				~& 80 & 1.15e-04& 3.01& 6.73e-05& 2.95&1.15e-04  & 3.01 & 6.73e-05 & 2.95\\
 				~& 160 & 1.43e-05& 3.00& 8.57e-06& 2.97&1.43e-05  & 3.00 & 8.57e-06 & 2.97\\
 				~ & ~ & ~ & ~ & ~ & ~ & ~ & ~ & ~ & ~\\
 				3& 10 & 4.06e-03& --& 2.35e-03& -- &4.06e-03 & -- & 2.27e-03 & --\\
 				~& 20 & 2.49e-04& 4.03& 1.39e-04& 4.08& 2.49e-04 & 4.03 & 1.44e-04 & 3.98\\
 				~& 40 & 1.55e-05& 4.01& 8.87e-06& 3.97&1.55e-05  & 4.01 & 8.87e-06 & 4.02\\
 				~& 80 & 9.67e-07& 4.00& 5.54e-07& 4.00&9.67e-07  & 4.00 & 5.54e-07 & 4.00\\
 				~& 160 & 6.06e-08& 4.00& 3.50e-08& 3.98&6.06e-08  & 4.00 & 3.57e-08 & 3.96\\
 				\hline
 			\end{tabular}
 		}
 	\end{center}
 	\caption{\scriptsize{We present the $L^2/L^\infty$ errors and the corresponding convergence rates for $u$ (the real part $\mbox{Re}(u)$ and the imaginary part $\mbox{Im}(u)$) for problem (\ref{eg1}) using $\mathcal{P}^q$ polynomials.  The interval is divided into $N$ uniform cells, and the terminal computational time $T = 1$.}  These results present the optimal convergence which is robust to the error norm.}\label{linear_u}
 \end{table}

 \begin{table}
 	\footnotesize
 	\begin{center}
 		\scalebox{1.0}{
 			\begin{tabular}{c c c c c c c c c c}
 				\hline
 				~ & ~ &  $\mbox{Re}(w)$ &~ &~&~& $\mbox{Im}(w)$& ~ & ~ & ~\\
 				\cline{3-6} \cline{7-10}
 				$q$ & $N$ & $L^2$ error & order & $L^\infty$ error & order  & $L^2$ error & order & $L^\infty$ error & order \\
 				\hline
 				1& 10 & 1.37e-00& --& 5.49e-01& -- &1.37e-00 & -- & 5.76e-01 & --\\
 				~& 20 & 3.19e-01& 2.10& 1.58e-01& 1.79& 3.19e-01 & 2.10 & 1.54e-01 & 1.90\\
 				~& 40 & 9.10e-02& 1.81& 4.33e-02& 1.87&9.10e-02  & 1.81 & 4.33e-02 & 1.83\\
 				~& 80 & 2.39e-02& 1.93& 1.11e-02& 1.96&2.39e-02  & 1.93 & 1.11e-02 & 1.96\\
 				~& 160 & 6.08e-03&1.97& 2.80e-03& 1.99 &6.08e-03  & 1.97 & 2.80e-03 & 1.99\\
 				~ & ~ & ~ & ~ & ~ & ~ & ~ & ~ & ~  & ~\\
 				2& 10 & 7.39e-02& --& 4.92e-02& -- &7.39e-02 & -- & 4.74e-02 & --\\
 				~& 20 & 7.58e-03& 3.29& 5.01e-03& 3.30& 7.58e-03 & 3.29 & 4.83e-03 & 3.29\\
 				~& 40 & 9.24e-04& 3.04& 5.95e-04& 3.07&9.24e-04  & 3.04 & 5.95e-04 & 3.02\\
 				~& 80 & 1.15e-04& 3.01& 7.22e-05& 3.04&1.15e-04  & 3.01 & 7.22e-05 & 3.04\\
 				~& 160 & 1.43e-05&3.00& 8.88e-06& 3.02&1.43e-05  & 3.00 & 8.88e-06 & 3.02\\
 				~ & ~ & ~ & ~ & ~ & ~ & ~ & ~ & ~ & ~\\
 				3& 10 & 4.06e-03& --& 2.43e-03& -- &4.06e-03 & -- & 2.50e-03 & --\\
 				~& 20 & 2.49e-04& 4.03& 1.44e-04& 4.08& 2.49e-04 & 4.03 & 1.43e-04 & 4.13\\
 				~& 40 & 1.55e-05& 4.01& 8.91e-06& 4.02 &1.55e-05  & 4.01 & 8.91e-06 & 4.00\\
 				~& 80 & 9.67e-07& 4.00& 5.54e-07& 4.01&9.67e-07  & 4.00 & 5.54e-07 & 4.01\\
 				~& 160 & 6.06e-08& 4.00& 3.57e-08& 3.96&6.06e-08  & 4.00 & 3.53e-08 & 3.97\\
 				\hline
 			\end{tabular}
 		}
 	\end{center}
 	\caption{\scriptsize{This table presents the errors and the corresponding convergence rates for $w$ in the problem (\ref{eg1}).  All are chosen to be the same as in Table \ref{linear_u}.}}\label{linear_w}
 \end{table}


 The numerical mass and the numerical Hamiltonian trajectories of the proposed ultra-weak LDG scheme for the problem (\ref{eg1}) are presented in Figure \ref{mass_hamitonian_linear} with both SDC and Crank--Nicolson time integrators.  In particular, we show the results for the approximation degree $q = 2$ until the final time $T = 100$ with $N = 40$. We note that the numerical mass and Hamiltonian are conserved by the conservative scheme (Crank--Nicolson time integrator).  Though the numerical mass and Hamiltonian are not conserved by the SDC time integrator, the magnitude of the numerical mass error is smaller than $10^{-5}$ and the numerical Hamiltonian error is smaller than $10^{-3}$.

\begin{figure}[htbp]
	\centering
	\includegraphics[width=0.45\textwidth]{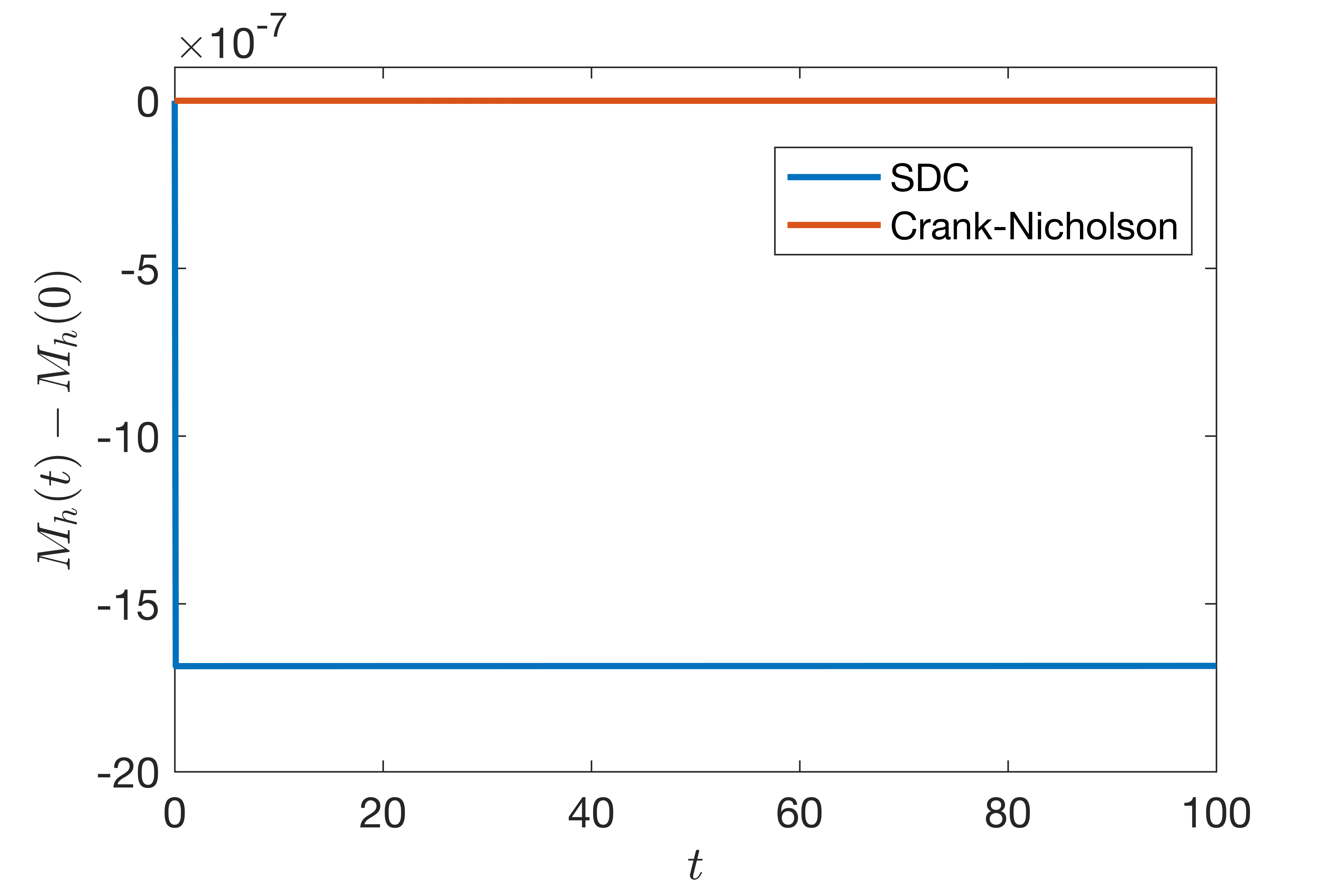}
	\includegraphics[width=0.45\textwidth]{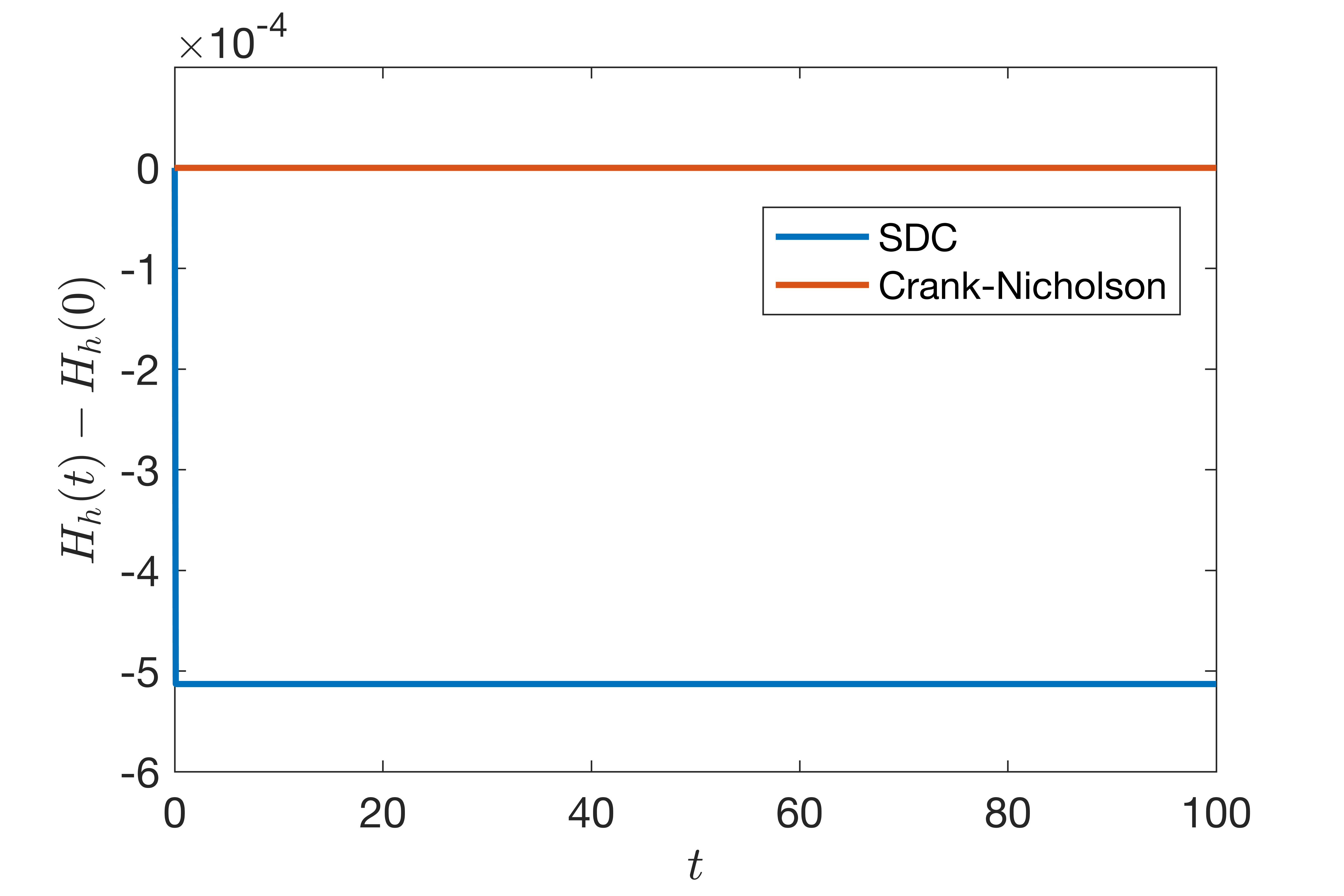}
	\caption{The left plots $M_h(t)-M_h(0)$ for problem (\ref{eg1}) using $\mathcal{P}^2$ polynomial on a uniform mesh of $N = 40$ up to a terminal time $T=100$, while the right plots $H_h(t)-H_h(0)$ under the same setting.
	}\label{mass_hamitonian_linear}
\end{figure}

 
 \subsection{Defocusing nonlinear problem in one dimensional space}\label{defocusing_sec}
  We provide another set of studies that examine the effectiveness and theoretical convergence order of the proposed ultra-weak LDG scheme for a defocusing nonlinear biharmonic Schr\"{o}dinger equation with $F'(|u|^2) = e^{|u|^2}$, that is,
 \begin{equation}\label{defocusing}
 iu_{t} + u_{xxxx} + ue^{|u|^2} = f(x,t),\ \ (x,t)\in(0, 4\pi)\times(0,1],
 \end{equation}
 subject to periodic boundary conditions and with initial data, external forcing $f(x,t)$ such that the exact solution is given by
 \[u(x,t) = \cos(x + t) + i\sin(x + t).\]
 We use the same spatial discretization as those in Section \ref{liner_sec} and display the $L^2$ and $L^\infty$ errors for the real part and the imaginary part of both $u$ and $w$ from Tables \ref{u_defocusing} to Table \ref{w_defocusing}. We observe similar results as those for the linear biharmonic Schr\"{o}dinger equations in Section \ref{liner_sec}. Specifically, we note an optimal convergence $q + 1$ for both the real part and the imaginary part of $u$ and $w$. 
 

Finally, Figure \ref{uw_non} presents the snapshots of the numerical solution $u_h$ and $w_h$ at $t = 1000$. Here, we choose the approximation degree $q = 2$ and the number of cells $N = 80$.  From this figure, we observe that our numerical solutions match very well with the exact solution.

  \begin{table}
 	\footnotesize
 	\begin{center}
 		\scalebox{1.0}{
 			\begin{tabular}{c c c c c c c c c c c}
 				\hline
 				~ & ~ &  $\mbox{Re}(u)$ &~ &~ &~&~& $\mbox{Im}(u)$& ~ & ~ & ~\\
 				\cline{3-6} \cline{8-11}
 				$q$ & $N$ & $L^2$ error & order & $L^\infty$ error & order & ~ & $L^2$ error & order & $L^\infty$ error & order \\
 				\hline
 				1& 10 & 5.30e-01& --& 2.10e-01& -- & ~ &5.30e-01 & -- & 2.18e-01 & --\\
 				~& 20 & 1.45e-01& 1.88& 5.72e-02& 1.87& ~& 1.45e-01 & 1.88 & 5.75e-02 & 1.92\\
 				~& 40 & 3.96e-02& 1.87& 1.76e-02& 1.70& ~ &3.96e-02  & 1.87 & 1.76e-02 & 1.71\\
 				~& 80 & 1.02e-02& 1.96& 4.72e-03& 1.90& ~ &1.02e-02  & 1.96 & 4.72e-03 & 1.90\\
 				~& 160 & 2.56e-03& 1.99& 1.21e-03& 1.97& ~ &2.56e-03  & 1.99 & 1.21e-03 & 1.97\\
 				~ & ~ & ~ & ~ & ~ & ~ & ~ & ~ & ~ & ~ & ~\\
 				2& 10 & 5.79e-02& --& 3.69e-02& -- & ~ &5.79e-02 & -- & 3.57e-02 & --\\
 				~& 20 & 7.39e-03& 2.97& 4.47e-03& 3.05& ~& 7.39e-03 & 2.97 & 4.66e-03 & 2.94\\
 				~& 40 & 9.18e-04& 3.01& 5.65e-04& 2.98& ~ &9.18e-04  & 3.01 & 5.65e-04 & 3.04\\
 				~& 80 & 1.15e-04& 3.00& 7.04e-05& 3.01& ~ &1.15e-04  & 3.00 & 7.04e-05 & 3.01\\
 				~& 160 & 1.43e-05& 3.00& 8.76e-06& 3.01& ~ &1.43e-05  & 3.00 & 8.76e-06 & 3.01\\
 				~ & ~ & ~ & ~ & ~ & ~ & ~ & ~ & ~ & ~ & ~\\
 				3& 10 & 3.98e-03& --& 2.31e-03& -- & ~ &3.98e-03 & -- & 2.37e-03 & --\\
 				~& 20 & 2.49e-04& 4.00& 1.44e-04& 4.01& ~& 2.49e-04 & 4.00& 1.38e-04 & 4.11\\
 				~& 40 & 1.55e-05& 4.00& 8.81e-06& 4.03& ~ &1.55e-05  & 4.00 & 8.81e-06 & 3.97\\
 				~& 80 & 9.67e-07& 4.00& 5.54e-07& 3.99& ~ &9.67e-07  & 4.00 & 5.54e-07 & 3.99\\
 				~& 160 & 6.05e-08& 4.00& 3.50e-08& 3.98& ~ &6.05e-08  & 4.00 & 3.49e-08 & 3.99\\
 				\hline
 			\end{tabular}
 		}
 			\end{center}
 	\caption{\scriptsize{Errors and the corresponding convergence rates for $u$ of problem (\ref{defocusing}) using $\mathcal{P}^q$ polynomials on a uniform mesh of $N$ cells up to terminal time $T = 1$.  This table adds additional evidence that supports the optimal and robust convergence of the proposed scheme.}}\label{u_defocusing}
 \end{table}
 
\begin{table}
	\footnotesize
	\begin{center}
		\scalebox{1.0}{
			\begin{tabular}{c c c c c c c c c c c}
				\hline
				~ & ~ &  $\mbox{Re}(w)$ &~ &~ &~&~& $\mbox{Im}(w)$& ~ & ~ & ~\\
				\cline{3-6} \cline{8-11}
				$q$ & $N$ & $L^2$ error & order & $L^\infty$ error & order & ~ & $L^2$ error & order & $L^\infty$ error & order \\
				\hline
				1& 10 & 9.52e-01& --& 4.09e-01& -- & ~ &9.52e-00 & -- & 3.94e-01 & --\\
				~& 20 & 1.90e-01& 2.33& 9.52e-02& 2.10& ~& 1.90e-01 & 2.32 & 9.06e-02 & 2.12\\
				~& 40 & 5.07e-02& 1.90& 2.64e-02& 1.84& ~ &5.07e-02  & 1.90 & 2.64e-02 & 1.78\\
				~& 80 & 1.28e-02& 1.99& 6.80e-03& 1.96& ~ &1.28e-02  & 1.99 & 6.80e-02 & 1.96\\
				~& 160 & 3.18e-03& 2.01& 1.72e-03& 1.99& ~ &3.18e-03  & 2.01 & 1.72e-03 & 1.99\\
				~ & ~ & ~ & ~ & ~ & ~ & ~ & ~ & ~ & ~ & ~\\
				2& 10 & 4.17e-02& --& 1.73e-02& -- & ~ &4.17e-02 & -- & 1.74e-02 & --\\
				~& 20 & 6.98e-03& 2.58& 3.82e-03& 2.18& ~& 6.98e-03 & 2.58 & 3.91e-03 & 2.16\\
				~& 40 & 9.06e-04& 2.94& 5.35e-04& 2.84& ~ &9.06e-04  & 2.94 & 5.35e-04 & 2.87\\
				~& 80 & 1.14e-04& 3.00& 6.89e-05& 2.96& ~ &1.14e-04  & 3.00 & 6.89e-05 & 2.96\\
				~& 160 & 1.43e-05& 3.00& 8.68e-06& 2.99& ~ &1.43e-05  & 3.00 & 8.68e-06 & 2.99\\
				~ & ~ & ~ & ~ & ~ & ~ & ~ & ~ & ~ & ~ & ~\\
				3& 10 & 3.05e-03& --& 1.47e-03& -- & ~ &3.05e-03 & -- & 1.51e-03 & --\\
				~& 20 & 2.29e-04& 3.73& 1.24e-04& 3.57& ~& 2.29e-04 & 3.73& 1.23e-04 & 3.61\\
				~& 40 & 1.52e-05& 3.92& 8.57e-06& 3.85& ~ &1.52e-05  & 3.92 & 8.57e-06 & 3.84\\
				~& 80 & 9.62e-07& 3.98& 5.49e-07& 3.96& ~ &9.62e-07  & 3.98 & 5.49e-07 & 3.96\\
				~& 160 & 6.04e-08& 3.99& 3.48e-08& 3.98& ~ &6.03e-08  & 3.99 & 3.48e-08 & 3.98\\
				\hline
			\end{tabular}
		}
	\end{center}
	\caption{\scriptsize{Errors and the corresponding convergence rates for $w$,  both real part $\mbox{Re}(w)$ and imaginary part $\mbox{Im}(w)$, in problem (\ref{defocusing}) when using $\mathcal{P}^q$ polynomials on a uniform mesh of $N(=80)$ cells.}}\label{w_defocusing}
\end{table}


\begin{figure}[htbp]
	\centering
	\includegraphics[width=0.45\textwidth]{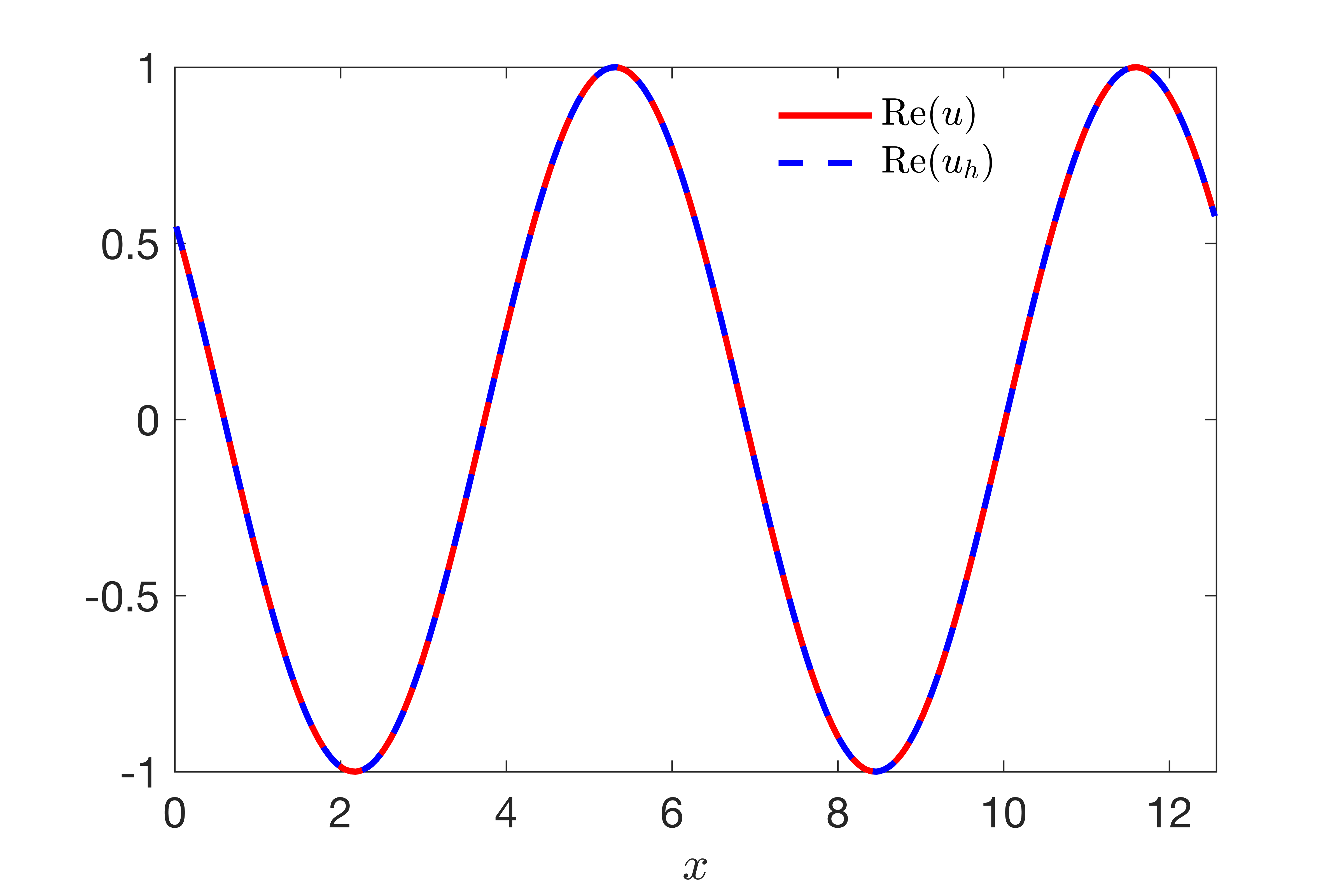}
	\includegraphics[width=0.45\textwidth]{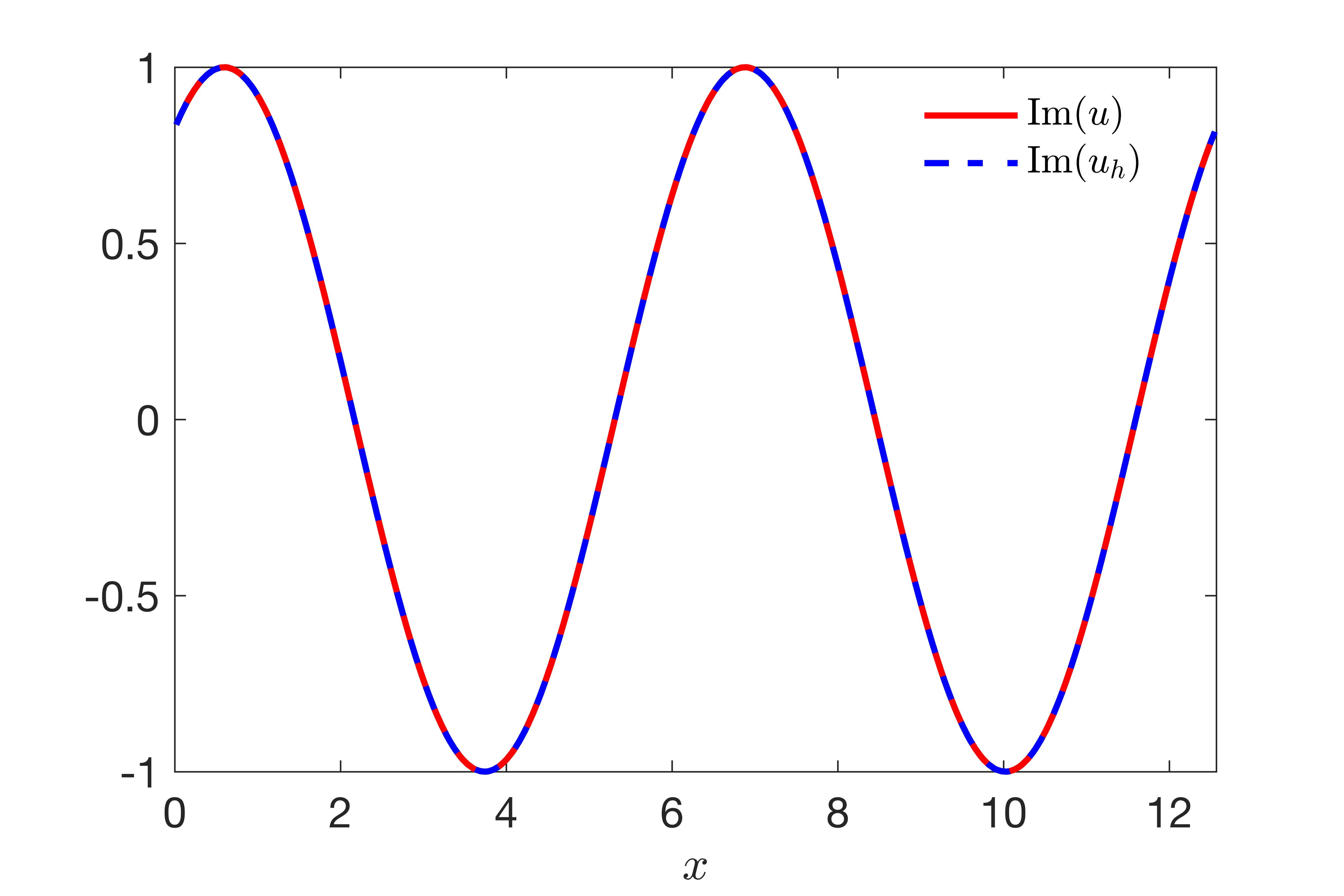}\\
	\includegraphics[width=0.45\textwidth]{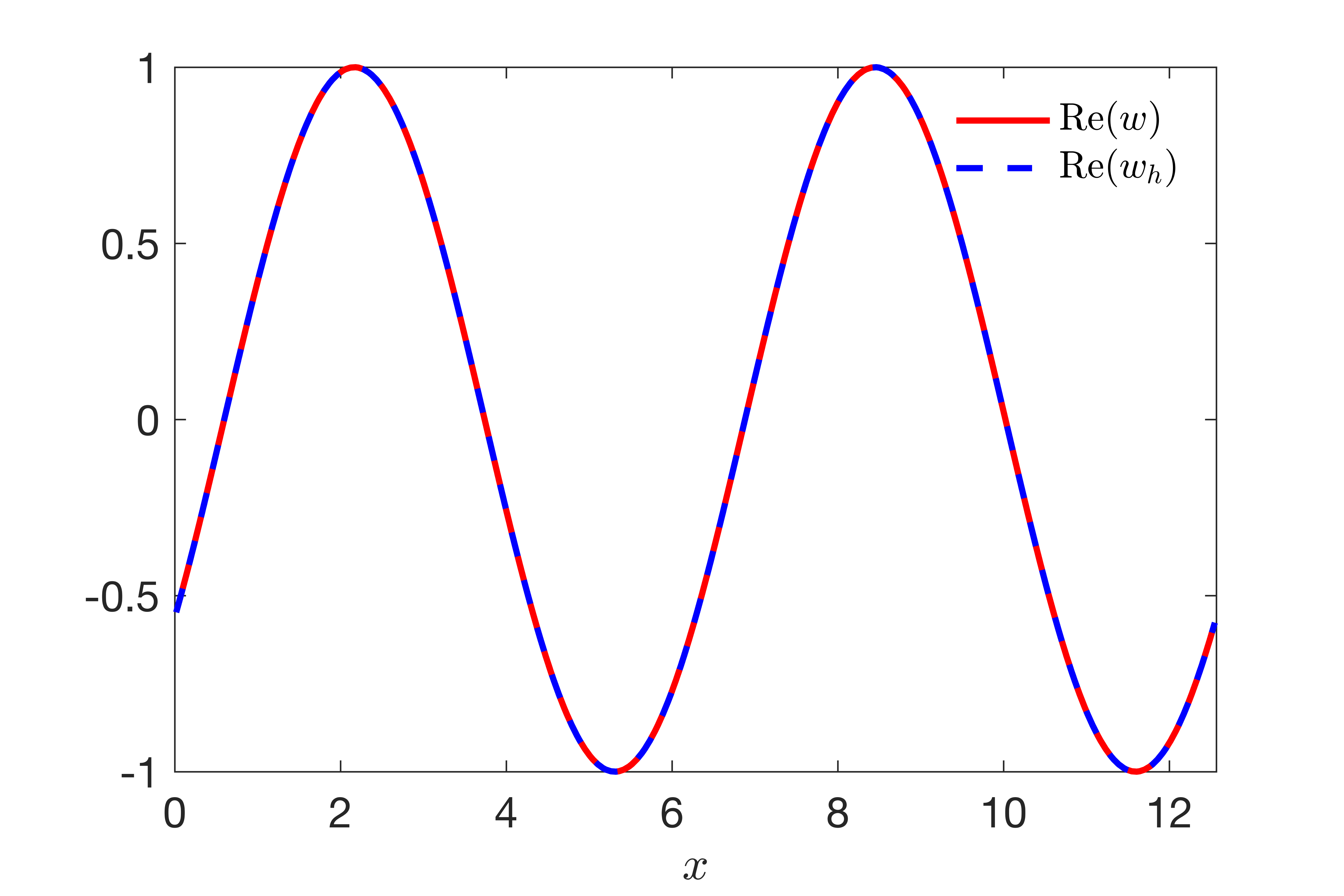}
	\includegraphics[width=0.45\textwidth]{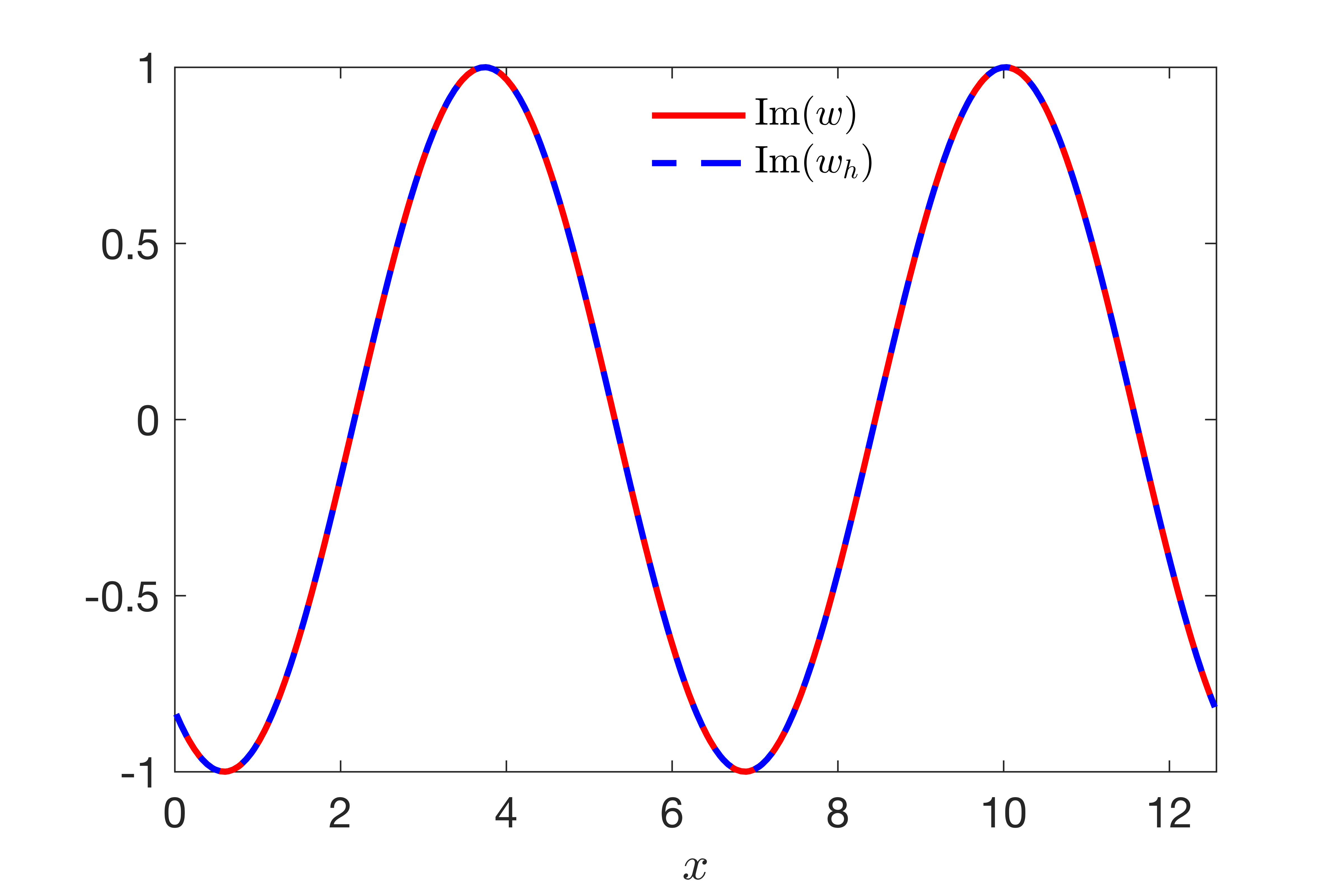}
	\caption{Snapshots of numerical and exact solutions of problem (\ref{defocusing}) at time $t = 1000$, where the numerical solutions are obtained using $\mathcal{P}^2$ polynomial on a uniform mesh of $N = 80$.  These plots provide evident support for the proposed scheme on approximating the exact solution, even up to a significantly long time.
	}\label{uw_non}
\end{figure}

\subsection{Focusing nonlinear problem in one dimensional space} 
 
 We now provide yet another set of experiments by considering nonlinear biharmonic Schr\"{o}dinger equations with an indefinite Hamiltonian.  To be specific, we test the problem
 \begin{equation*}
 iu_{t} = -u_{xxxx} - uF'(|u|^2), \ \ x\in(0,4\pi),\ \ t>0
 \end{equation*}
under two different nonlinear media, $F'(|u|^2)$,
 \[(a).\  F'(|u|^2) = -|u|^2, \ \ \ (b). \ F'(|u|^2) = -|u|^4.\]
Again, we study these problems with the same periodic boundary conditions and the following initial data as in Section \ref{defocusing_sec}
 \[u(x,0) = \cos x + i\sin x.\]
Finally, we also use the same spatial discretization as in Section \ref{liner_sec} with approximation order $q = 2$ and the number of cells $N = 80$.  Figure \ref{focusing} presents the temporal dynamics of the discrete solution $u_h$ under these different nonlinear media until $T = 100$.   From the top to the bottom we choose $F'(u) = -|u|^2$ and $F'(u) = -|u|^4$, respectively. On the top panel, we find that $u_h$ is still stable up to $t \approx 30$, however, its dynamics revolve after then and a stable time-periodic profile develops afterward.  On the bottom panel, we note a new stable time-periodic profile develops from the original solution around $t \approx 13$.
 
 \begin{figure}[htbp]
 	\centering
 	\includegraphics[width=0.45\textwidth]{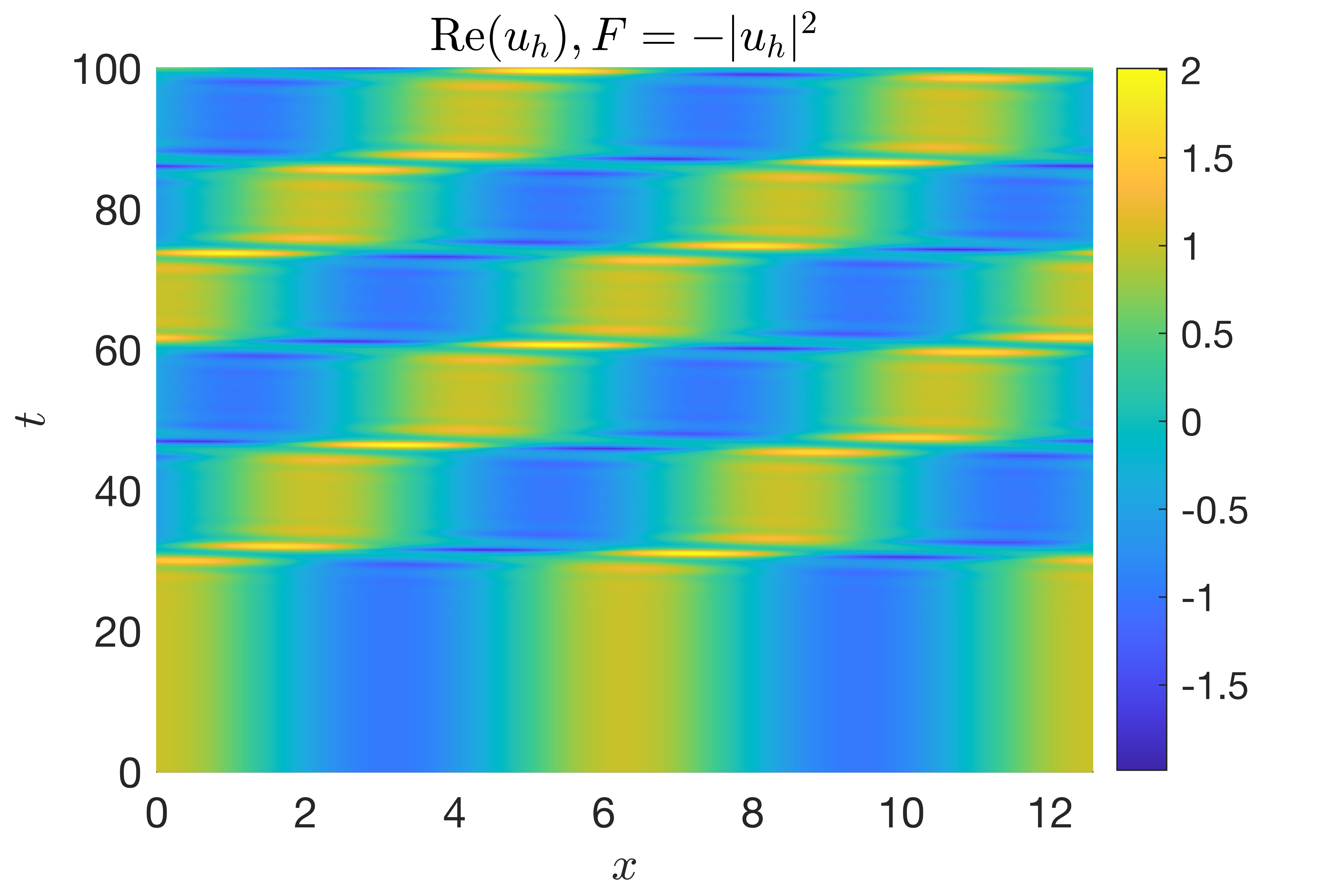}
 	\includegraphics[width=0.45\textwidth]{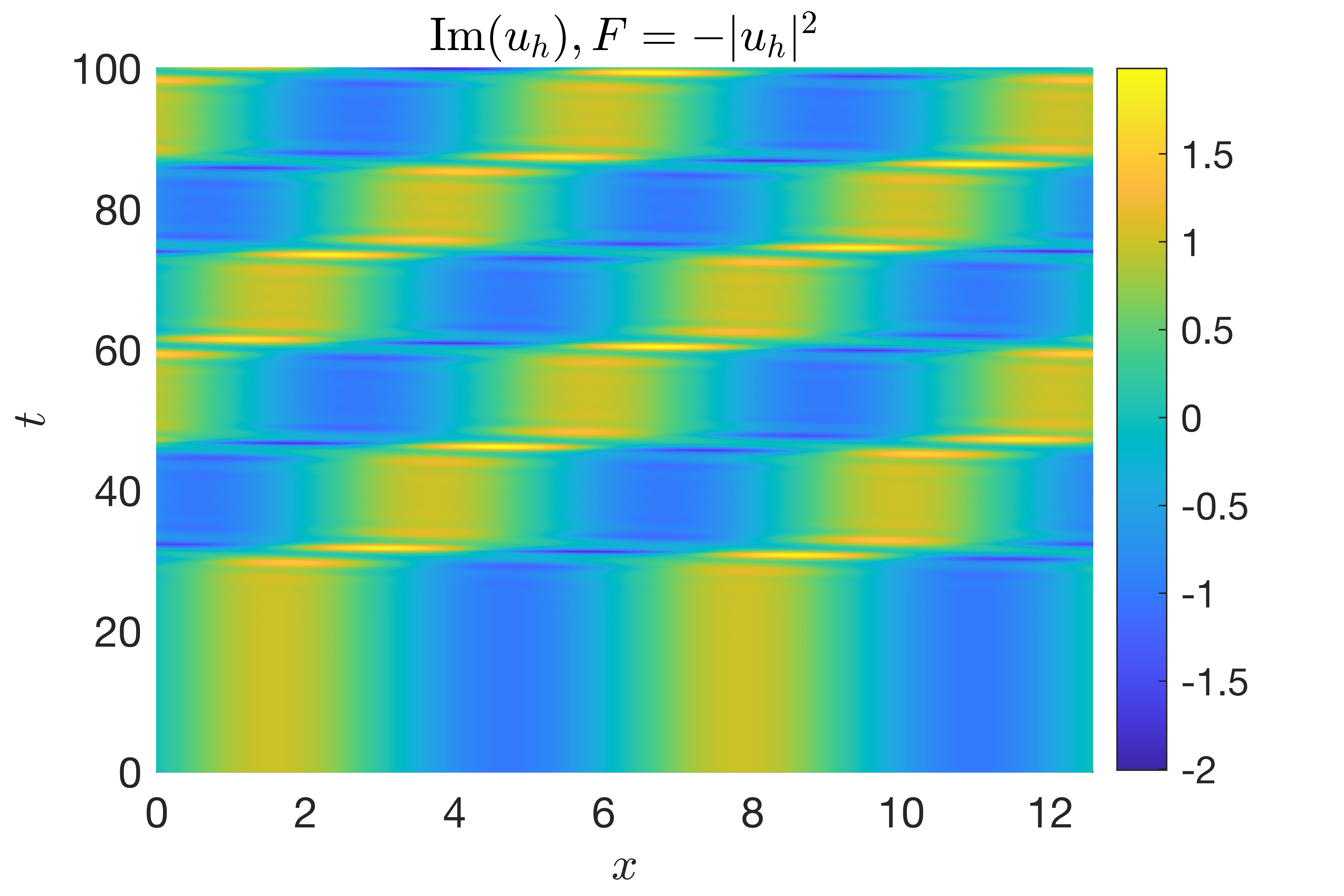}\\
 		\includegraphics[width=0.45\textwidth]{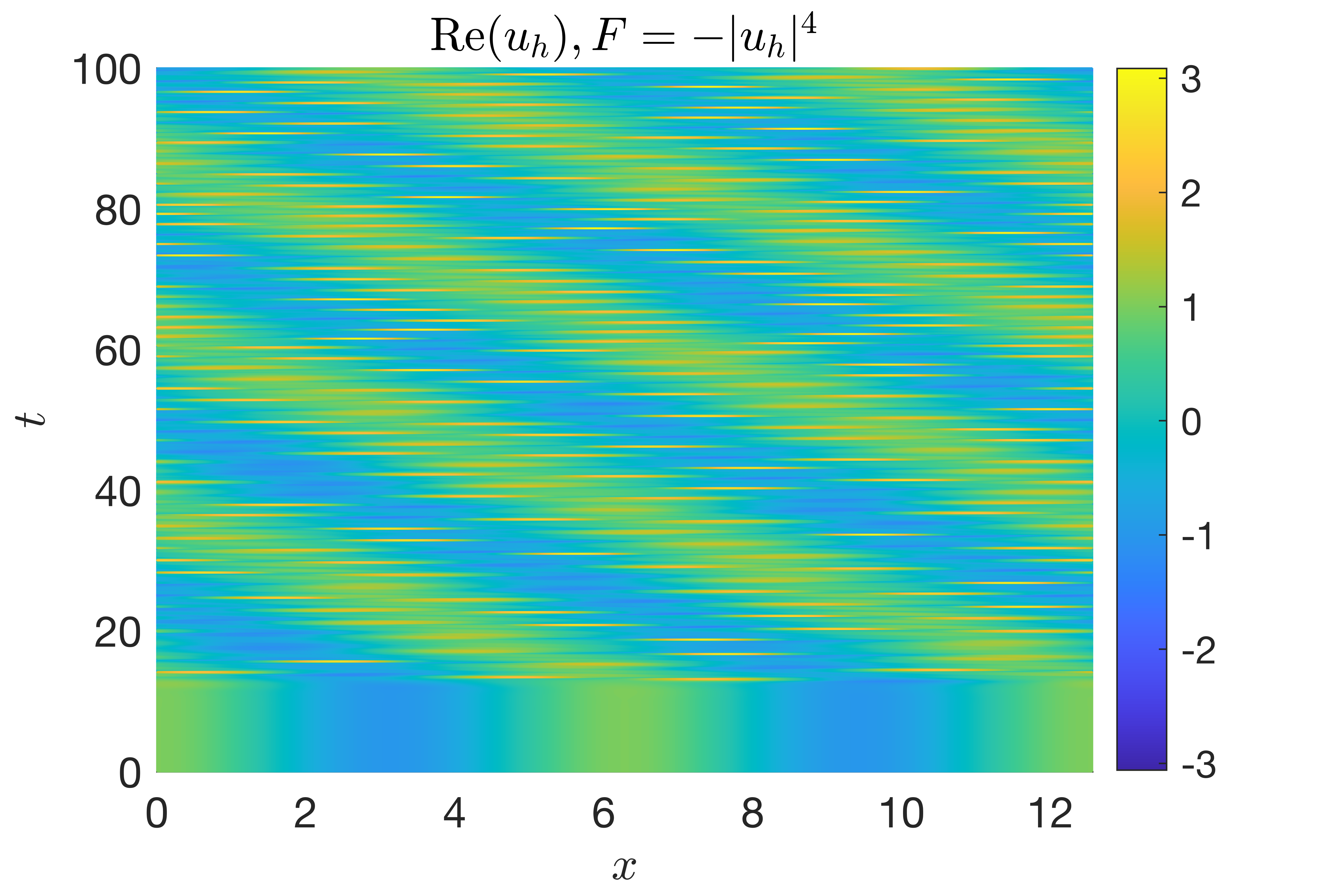}
 	\includegraphics[width=0.45\textwidth]{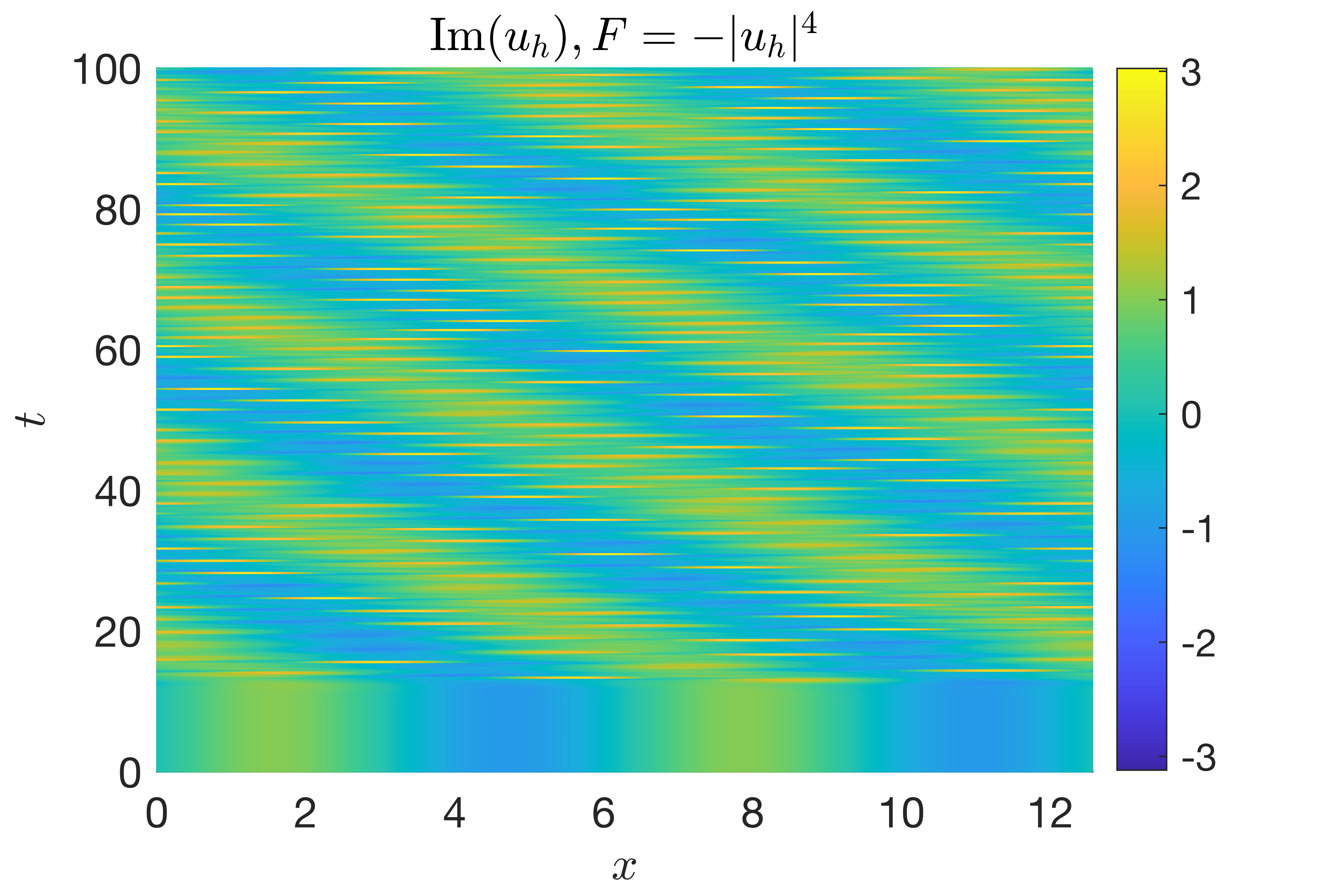}
 	\caption{Spatial-temporal dynamics of the numerical solutions.  Here the non-linearity of the medium are chosen to be $F'(|u_h|^2) = -|u_h|^2$ and $F'(|u_h|) = -|u_h|^4$ from the top to the bottom.  These plots readily indicate that the spatial-temporal dynamics of the problem heavily depend on the choice of the medium on one hand, and they are extremely complex which include oscillating.  {\textbf{Top}}:  the real part of $u_h$ (left) and the imaginary part of $u_h$ (right) with $F'(|u_h|^2) = -|u_h|^2$. {\textbf{Bottom}}: the real part of $u_h$ (left) and the imaginary part of $u_h$ (right) with $F'(|u_h|^2) = -|u_h|^4$. 
 	}\label{focusing}
 \end{figure}
 
 \subsection{Defocusing nonlinear problem in two dimensional space}
 In this example, we investigate the convergence of the ultra-weak LDG scheme for the nonlinear Schr\"{o}dinger equation with $F'(|u|^2) = |u|^2$ in two space dimensions. Precisely we solve
 \begin{equation}\label{2d_problem}
 iu_t + \Delta^2 u + u|u|^2 = 0, \quad (x,y,t) \in (0,2\pi)\times(0,2\pi)\times(0,1],
 \end{equation}
 with periodic boundary conditions and initial data
 \[u(x,y,0) = \cos(x+y) + i\sin(x+y).\]
 This yields the following exact solution
 \[u(x,y,t) = \cos(x+y+5t) + i\sin(x+y+5t).\]
 
The discretization is performed with elements over the Cartesian grids formed by $(x_k,y_j) = (kh,jh), k,j = 0,1,...,N$ with $h = 2\pi/N$.  Here, we only present the results for $u$, since the results for both $u$ and $w$ are similar to the problems in one space dimension. Table \ref{u_2d} displays the $L^2$ and $L^\infty$ errors for the real and the imaginary part of $u$. We observe optimal convergence for both cases.

 \begin{table}
 	\footnotesize
 	\begin{center}
 		\scalebox{1.0}{
 			\begin{tabular}{c c c c c c c c c c c}
 				\hline
 				~ & ~ &  $\mbox{Re}(u)$ &~ &~ &~&~& $\mbox{Im}(u)$& ~ & ~ & ~\\
 				\cline{3-6} \cline{8-11}
 				$q$ & $N\times N$ & $L^2$ error & order & $L^\infty$ error & order & ~ & $L^2$ error & order & $L^\infty$ error & order \\
 				\hline
 				1& $4\times 4$ & 7.40e-00& --& 2.17e-00& -- & ~ &7.40e-00 & -- & 2.17e-00 & --\\
 				~& $8\times 8$ & 3.28e-00& 1.17& 8.34e-01& 1.38& ~& 3.28e-00 & 1.17 & 8.34e-01 & 1.38\\
 				~& $16 \times 16$ & 9.05e-01& 1.86& 2.16e-01& 1.95& ~ &9.05e-01  & 1.86 & 2.16e-01 & 1.95\\
 				~& $32 \times 32$ & 2.30e-01& 1.98& 5.38e-02& 2.00& ~ &2.30e-01  & 1.98 & 5.38e-02 & 2.00\\
 				~ & ~ & ~ & ~ & ~ & ~ & ~ & ~ & ~ & ~ & ~\\
 				2& $4\times 4$ & 6.25e-01& --& 1.98e-01& -- & ~ &6.25e-01 & -- & 1.98e-01 & --\\
 				~& $8\times 8$ & 4.78e-02& 3.71& 1.67e-02& 3.58& ~& 4.78e-02 & 3.71 & 1.67e-02 & 3.57\\
 				~& $16\times 16$ & 4.91e-03& 3.28& 1.79e-03& 3.23& ~ &4.91e-03  & 3.28 & 1.79e-03 & 3.23\\
 				~& $32\times 32$ & 6.12e-04& 3.01& 2.22e-04& 3.01& ~ &6.12e-04  & 3.01 & 2.22e-04 & 3.01\\
 				~ & ~ & ~ & ~ & ~ & ~ & ~ & ~ & ~ & ~ & ~\\
 				3& $4\times 4$ & 2.73e-02& --& 1.09e-02& -- & ~ &2.73e-02 & -- & 1.09e-02 & --\\
 				~& $8\times 8$ & 1.55e-03& 4.14& 7.05e-04& 3.95& ~& 1.55e-03 & 4.14& 7.05e-04 & 3.95\\
 				~& $16\times 16$ & 9.54e-05& 4.02& 4.42e-05& 3.99& ~ &9.54e-05  & 4.02 & 4.42e-05 & 3.99\\
 				~& $32\times 32$ & 5.93e-06& 4.01& 2.73e-06& 4.02& ~ &5.93e-06  & 4.01 & 2.73e-06 & 4.02\\
 				\hline
 			\end{tabular}
 		}
 			\end{center}
 	\caption{\scriptsize{Errors and the corresponding convergence rates for $u$ of problem (\ref{2d_problem}) using $\mathcal{Q}^q$ polynomials on a uniform Cartesian mesh of $N\times N$ elements up to terminal time $T = 1$.}}\label{u_2d}
 \end{table} 
 
 \subsection{Mixed boundary condition in two dimensional space}
Lastly, we consider the nonlinear biharmonical Schr\"{o}dinger equation
\begin{equation}\label{2d_problem_bc}
 iu_t + \Delta^2 u + u(2 + \sin(|u|^2)) = f(x,y,t), \quad (x,y,t) \in (0,1)\times(0,1)\times(0,1],
 \end{equation}
 with the following mixed boundary conditions, 
\[  \mbox{top boundary:} \left\{
 \begin{aligned}
u(x, 2\pi, t) &= 0,\\
u_y(x, 2\pi, t) &= 0,
\end{aligned}
\right.\quad \mbox{bottom boundary:}
\left\{
\begin{aligned}
u_{yy}(x, 0, t) &= 0,\\
u_{yyy}(x, 0, t) &= 0,
\end{aligned}
\right.\]
\[\mbox{right boundary:}
\left\{
\begin{aligned}
u(1, y, t) &= u(0, y, t),\\
u_x(1, y, t) &= u_x(0,y,t),
\end{aligned}
\right.\quad \mbox{left boundary:}\left\{
\begin{aligned}
u_{xx}(0, y, t) &= u_{xx}(1, y, t),\\
u_{xxx}(0, y, t) &= u_{xxx}(1, y, t),
\end{aligned}
\right.
\]
and the exact solution
\begin{equation}\label{sol_2d_bc}
u(x,y,t) = y^4(y - 1)^4(\cos(2\pi x) + i\sin(2\pi x)) e^t.
\end{equation}
Then the external forcing $f(x,y,t)$ is obtained by solving (\ref{2d_problem_bc}) with (\ref{sol_2d_bc}).

\begin{table}
 	\footnotesize
 	\begin{center}
 		\scalebox{1.0}{
 			\begin{tabular}{c c c c c c c c c c c}
 				\hline
 				~ & ~ &  $\mbox{Re}(u)$ &~ &~ &~&~& $\mbox{Im}(u)$& ~ & ~ & ~\\
 				\cline{3-6} \cline{8-11}
 				$q$ & $N\times N$ & $L^2$ error & order & $L^\infty$ error & order & ~ & $L^2$ error & order & $L^\infty$ error & order \\
 				\hline
 				1& $15\times 15$ & 1.78e-02& --& 8.47e-02& --& ~& 1.78e-02 & -- & 8.47e-02 & --\\
 				~& $20 \times 20$ & 7.33e-03& 3.09& 3.50e-02& 3.07& ~ &7.33e-03  & 3.09 & 3.50e-02 & 3.06\\
 				~& $25 \times 25$ & 3.67e-03& 3.10& 1.77e-02& 3.06& ~ &3.67e-03  & 3.10 & 1.77e-02 & 3.07\\
 				~& $30 \times 30$ & 2.08e-03& 3.11& 1.01e-02& 3.08& ~ &2.08e-03  & 3.11 & 1.01e-02 & 3.08\\
 				~& $35 \times 35$ & 1.29e-03& 3.10& 6.27e-03& 3.09& ~ &1.29e-03  & 3.10 & 6.27e-03 & 3.09\\
 				~ & ~ & ~ & ~ & ~ & ~ & ~ & ~ & ~ & ~ & ~\\
 				2& $15\times 15$ & 3.29e-04& --& 2.85e-03& --& ~& 3.29e-04 & -- & 2.85e-03 & --\\
 				~& $20\times 20$ & 1.17e-04& 3.60& 1.04e-03& 3.50& ~ &1.17e-04  & 3.60 & 1.04e-03 & 3.50\\
 				~& $25\times 25$ & 5.18e-05& 3.65& 5.12e-04& 3.18& ~ &5.18e-05  & 3.65 & 5.11e-04 & 3.18\\
 				~& $30\times 30$ & 2.64e-05& 3.69& 2.94e-04& 3.04& ~ &2.64e-05  & 3.69 & 2.93e-04 & 3.06\\
 				~& $35\times 35$ & 1.50e-05& 3.67& 1.85e-04& 3.00& ~ &1.50e-05  & 3.67 & 1.85e-04 & 2.98\\
 				~ & ~ & ~ & ~ & ~ & ~ & ~ & ~ & ~ & ~ & ~\\
 				3& $15\times 15$ & 8.88e-06& --& 1.86e-04& --& ~& 8.88e-06 & --& 1.87e-04 & --\\
 				~& $20\times 20$ & 2.59e-06& 4.29& 6.38e-05& 3.73& ~ &2.59e-06  & 4.29 & 6.38e-05 & 3.74\\
 				~& $25\times 25$ & 9.83e-07& 4.34& 2.73e-05& 3.81& ~ &9.83e-07  & 4.34 & 2.73e-05 & 3.81\\
 				~& $30\times 30$ & 4.43e-07& 4.37& 1.35e-05& 3.86& ~ &4.43e-07  & 4.37 & 1.35e-05 & 3.86\\
 				~& $35\times 35$ & 2.25e-07& 4.39& 7.42e-06& 3.88& ~ &2.25e-07  & 4.39 & 7.42e-06 & 3.88\\
 				\hline
 			\end{tabular}
 		}
 			\end{center}
 	\caption{\scriptsize{Errors and the corresponding convergence rates for $u$ of problem (\ref{2d_problem_bc}) using $\mathcal{Q}^q$ polynomials on a uniform Cartesian mesh of $N\times N$ elements up to terminal time $T = 1$.}}\label{u_2d_bc}
 \end{table} 
 Table \ref{u_2d_bc} presents the $L^2/L^\infty$ errors of $u$ for the problem (\ref{2d_problem_bc}), while Table \ref{w_2d_bc} displays the $L^2/L^\infty$ errors of $w$. We observe optimal convergence for both $u$ and $w$ when $q = 2,3$. When $q = 1$, we note a super-convergence $q+2$ for $u$ in both $L^2$ and $L^\infty$; optimal convergence for $w$ in $L^2$, and a super-convergence $q+2$ for $L^\infty$.
 
 \begin{table}
 	\footnotesize
 	\begin{center}
 		\scalebox{1.0}{
 			\begin{tabular}{c c c c c c c c c c c}
 				\hline
 				~ & ~ &  $\mbox{Re}(w)$ &~ &~ &~&~& $\mbox{Im}(w)$& ~ & ~ & ~\\
 				\cline{3-6} \cline{8-11}
 				$q$ & $N\times N$ & $L^2$ error & order & $L^\infty$ error & order & ~ & $L^2$ error & order & $L^\infty$ error & order \\
 				\hline
 				1& $15\times 15$ & 7.10e-03& --& 3.22e-02& --& ~& 7.10e-03 & -- & 3.20e-02 & --\\
 				~& $20 \times 20$ & 2.94e-03& 3.07& 1.26e-02& 3.27& ~ &2.94e-03  & 3.07 & 1.26e-02 & 3.26\\
 				~& $25 \times 25$ & 1.56e-03& 2.83& 6.24e-03& 3.13& ~ &1.56e-03  & 2.83 & 6.23e-03 & 3.14\\
 				~& $30 \times 30$ & 9.61e-04& 2.66& 3.56e-03& 3.09& ~ &9.61e-04  & 2.66 & 3.55e-03 & 3.08\\
 				~& $35 \times 35$ & 6.51e-04& 2.53& 2.23e-03& 3.03& ~ &6.51e-04  & 2.53 & 2.23e-03 & 3.02\\
 				~ & ~ & ~ & ~ & ~ & ~ & ~ & ~ & ~ & ~ & ~\\
 				2& $15\times 15$ & 1.76e-04& --& 7.67e-04& --& ~& 1.76e-04 & -- & 7.67e-04 & --\\
 				~& $20\times 20$ & 7.49e-05& 2.96& 4.16e-04& 2.13& ~ &7.49e-05  & 2.96 & 4.16e-04 & 2.13\\
 				~& $25\times 25$ & 3.57e-05& 3.33& 1.67e-04& 4.10& ~ &3.57e-05  & 3.33 & 1.66e-04 & 4.11\\
 				~& $30\times 30$ & 2.04e-05& 3.08& 9.99e-05& 2.81& ~ &2.04e-05  & 3.08 & 9.99e-05 & 2.80\\
 				~& $35\times 35$ & 1.27e-05& 3.06& 6.45e-05& 2.84& ~ &1.27e-05  & 3.06 & 6.45e-05 & 2.84\\
 				~ & ~ & ~ & ~ & ~ & ~ & ~ & ~ & ~ & ~ & ~\\
 				3& $15\times 15$ & 5.92e-06& --& 2.28e-05& --& ~& 5.92e-06 & --& 2.27e-05 & --\\
 				~& $20\times 20$ & 1.88e-06& 4.00& 7.15e-06& 4.03& ~ &1.88e-06  & 4.00 & 7.15e-06 & 4.02\\
 				~& $25\times 25$ & 7.69e-07& 4.00& 2.91e-06& 4.02& ~ &7.69e-07  & 4.00 & 2.91e-06 & 4.03\\
 				~& $30\times 30$ & 3.71e-07& 4.00& 1.40e-06& 4.02& ~ &3.71e-07  & 4.00 & 1.40e-06 & 4.03\\
 				~& $35\times 35$ & 2.00e-07& 4.01& 7.55e-07& 4.01& ~ &2.00e-07  & 4.01 & 7.54e-07 & 4.01\\
 				\hline
 			\end{tabular}
 		}
 			\end{center}
 	\caption{\scriptsize{Errors and the corresponding convergence rates for $w$ of problem (\ref{2d_problem_bc}) using $\mathcal{Q}^q$ polynomials on a uniform Cartesian mesh of $N\times N$ elements up to terminal time $T = 1$.}}\label{w_2d_bc}
 \end{table}

\section{Brief Conclusions}\label{conclusion} 
In conclusion, we have developed and analyzed an ultra-weak LDG method for nonlinear biharmonic Schr\"{o}dinger equations in both one dimensional space and two dimensional space.  We extend the LDG scheme and introduce a second-order spatial derivative as an auxiliary variable.  This maneuver reduces the storage for the variables to be solved hence enhancing the computational efficiency. The scheme is also stable without employing any penalty term. We have proved and demonstrated the stability of the scheme for the special projection operators; moreover, we also obtain optimal $L^2$-error estimates under these settings. We also show that the fully discrete scheme combined with the Crank--Nicolson time integrator is conservative.  Our numerical experiments demonstrate the theoretical findings and present the rich and complex spatial-temporal dynamics of these linear/nonlinear problems.  

\section*{Acknowledgements} The author would like to thank Professor T. Hagstrom and Professor Q. Wang for very useful comments and discussions.

\bibliography{lu_new}
\bibliographystyle{plain}

\clearpage



\end{document}